\documentclass{article}

\usepackage{amsmath}
\usepackage{amsthm}
\usepackage{amssymb}
\usepackage{bm}
\usepackage{graphicx} 
\usepackage{booktabs} 
\usepackage{amsfonts, xfrac}
\usepackage[font=scriptsize]{subcaption}
\usepackage{multirow}
\usepackage[inline]{enumitem}
\usepackage{hyperref}
\usepackage{float}
\usepackage{systeme}
\usepackage{xcolor}
\usepackage{kbordermatrix}
\usepackage{algorithm}
\usepackage[noend]{algorithmic}
\usepackage[round, authoryear, sort]{natbib}

\usepackage[a4paper, total={6in, 8in}]{geometry}

\newtheorem{theorem}{Theorem}

\newtheorem{corollary}{Corollary}
\newtheorem{lemma}{Lemma}

\newtheorem{example}{Example}
\newtheorem{remark}{Remark}

\newcommand{\Z}{\mathbb{Z}}

\newcommand{\R}{\mathbb{R}}

\newcommand{\Q}{\mathbb{Q}}

\newcommand{\F}{\mathcal{F}}
\newcommand{\M}{\mathcal{M}}
\renewcommand{\P}{\mathcal{P}}

\renewcommand{\S}{\mathcal{S}}
\renewcommand{\L}{\mathcal{L}}

\renewcommand{\kbldelim}{(}
\renewcommand{\kbrdelim}{)}

\newcommand{\ie} {{\em i.e.\/}, }
\newcommand{\eg} {{\em e.g.\/}, }

\newcommand{\ceil}[1]{\lceil #1 \rceil}
\newcommand{\floor}[1]{\lfloor #1 \rfloor}

\title{A Hereditary Property of Cutting Plane Procedures\vspace{-0em}}
\author{G\'erard Cornu\'ejols, Vrishabh Patil}

\begin{document}

\maketitle

\begin{abstract}
     \noindent Let $K'$ denote the closure of a convex set $K$ under a given cutting-plane procedure. The procedure satisfies the \emph{hereditary property} if $F' = K' \cap F$ for every face $F$ of $K$. The property underlies inductive proofs of finite-rank and the polyhedrality of closures, and it is an admissibility requirement in abstract frameworks for cutting-plane procedures. Yet, it has not been studied systematically across well-known cutting-plane procedures. We establish two sufficient conditions for the property. The first applies to procedures that can be expressed as closures under intersection cuts from a family of lattice-free convex sets, when a single family realizes the closure of $K$ and of each of its faces. It yields the hereditary property for the split, lift-and-project, Lov\'asz--Schrijver, Sherali--Adams, and Lasserre closures, applied over general convex sets and for faces that need not be exposed. The second applies to procedures whose cuts are valid inequalities for Gomory's corner polyhedron, 
     and it yields the hereditary property for the closure of Dantzig cuts derived from all bases. We complement these results with four classical procedures for which the hereditary property fails, namely the closure of Gomory's fractional cuts (derived from either all bases or feasible bases only), the mixed-integer Chv\'atal closure, the $+$-cut closure, and the closure of Dantzig cuts derived from feasible bases. 
\end{abstract}


\section{Introduction}\label{sec:intro}

Cutting-plane procedures are among the most fundamental tools in integer and mixed-integer optimization. They provide a systematic mechanism for strengthening linear relaxations by iteratively adding linear inequalities that separate fractional solutions while preserving all integer feasible points. The resulting sequence of relaxations, or \emph{closures}, has proven central both to theoretical analyses and to the design of modern mixed-integer linear programming algorithms. In practice, cutting-plane methods complement enumeration-based algorithms such as branch-and-bound.

Despite the diversity among cutting plane procedures (Chv\'atal cuts, split cuts, lift‐and‐project cuts, hierarchies such as Lov\'asz‐Schrijver, Sherali–Adams, Lasserre, among others), one unifying feature that recurs in many finite‐rank and admissibility arguments is the following \emph{hereditary property}. Let $K \subseteq \R^n$ be a convex set, and let $K' \subseteq K$ denote the closure of $K$ under a given cutting plane procedure — that is, the convex set obtained by intersecting $K$ with all the linear inequalities derived by the procedure. In this paper, we assume that, if a cutting plane procedure can be applied to a convex set $K$, it can also be applied to all its faces. Recall that a convex subset $F \subseteq K$ is a \emph{face} of $K$ if $\lambda u + (1-\lambda) v \in F$ with $u, v \in K$ and $\lambda \in (0,1)$ implies $u, v \in F$. Every face $F$ of a convex set $K$ is also a convex set (see, e.g., \cite{rockafellar1970convex}, \cite{hiriart2004fundamentals}). A face $F$ is \emph{exposed} if there is an inequality $c x \leq \delta$, valid for $K$, such that $F = K \cap \{x : c x = \delta\}$. The faces of a polyhedron are always exposed, but this is not true in general for convex sets. A cutting plane procedure is said to satisfy the {\it hereditary property} if, for every face $F$ of $K$ (not necessarily exposed), the closure of $F$ under the same procedure satisfies
$$F' = K' \cap F.$$
 \cite{pokutta2010rank} use the term \emph{homogeneity} for the restricted property $(F \cap K)' = K' \cap F$ where $F$ is a face of the unit hypercube $[0,1]^n$ and $K \subseteq [0,1]^n$, while \cite{dey2014design} use \emph{strong homogeneity} for the more general property considered in this paper. We prefer the term \emph{hereditary property}, in reference to the face lattice of $K$ ordered by inclusion.

\subsection{Literature Review} \label{subsec:lit-review}

The hereditary property plays a pivotal role in several aspects of cutting-plane theory (see for example \cite{schrijver1986theory}, \cite{cook1990chvatal}). It underlies inductive proofs of finite-rank results, which decompose the analysis across faces of decreasing dimension, as in classical arguments bounding the rank and establishing polyhedrality of the Chv\'atal closure (\cite{chvatal1973edmonds}, \cite{dunkel2013gomory}, \cite{dadush2014chvatal}). It is also an essential condition for admissibility in abstract frameworks for closure procedures (\cite{pokutta2010rank}, \cite{dey2014design}).

Despite its fundamental nature, the hereditary property has not been systematically studied across the broad spectrum of cutting-plane procedures that have appeared in the literature. While the property has been proven to hold for some classical closures, such as the Chv\'atal closure and the Lov\'asz--Schrijver relaxation (\cite{dash2001matrix}), it remains formally unclear for other well-known cuts, such as split cuts and Dantzig cuts, and hierarchies, including the lift-and-project procedures of Balas-Ceria-Cornu\'ejols, Sherali–Adams, and Lasserre.
Moreover, counterexamples for other procedures are seldom discussed.
This gap motivates a systematic investigation of the hereditary property as a unifying concept in cutting-plane theory.

For the Chv\'atal closure (\cite{gomory1958outline}, \cite{chvatal1973edmonds}), the hereditary property is proven within the classical finite-rank analysis of \cite{schrijver1980cutting} and \cite{cook1990chvatal}. \cite{dunkel2013gomory} extended the hereditary property to irrational polytopes, and \cite{dadush2014chvatal} further extended it to arbitrary compact convex sets, with a short alternative proof given by \cite{braun2014short}.

Beyond the Chv\'atal framework, \cite{cook1990chvatal} introduced the \emph{split closure}, derived from disjunctions $\{x \in \R^n: \pi_0 \leq \pi x \leq \pi_0 + 1\}$, and viewed the Chv\'atal closure as a type of \emph{split-and-round} closure for which they showed the hereditary property. Their argument is restricted to the pure integer setting and does not naturally extend to the mixed-integer case. This property has not been formally proven for the classical split closure.

The hereditary property has been established for the Lov\'asz--Schrijver operator on rational polytopes (\cite{dash2001matrix}). Whether this property also holds for the lift-and-project closure (\cite{balas1993lift}), the Sherali--Adams hierarchy (\cite{sherali1990hierarchy}), and the Lasserre hierarchy (\cite{lasserre2001explicit}) has not been explicitly analyzed.

Our analysis builds on the intersection-cut framework of \cite{balas1971intersection}. \cite{averkov2012finitely} recasts intersection cuts in terms of $L$-reductions and closure operators, the formulation we adopt in this work. For procedures that are intrinsically tied to a polyhedral description, we rely instead on Gomory's corner polyhedron (\cite{gomory1969some}). The hereditary property has not been studied in either setting.

\subsection{Our Contributions and Outline} \label{subsec:contributions-outline}

This paper makes three contributions. The first, Theorem~\ref{thm:L-closure-hered}, establishes the hereditary property for any closure operator that can be expressed as an \emph{$\L$-closure}, provided the same family $\L$ of lattice-free convex sets realizes the closure of the underlying convex set $K$ and of each of its faces. This is precisely the setting of intersection cuts in the sense of \cite{balas1971intersection} and \cite{averkov2012finitely}, and we show that these cuts cover the split, lift-and-project, Lov\'asz--Schrijver, Sherali--Adams, and Lasserre procedures. The second contribution is a collection of counterexamples showing that the hereditary property fails for four classical procedures, namely the closure of Gomory fractional cuts, derived from either all bases or feasible bases only, the mixed-integer Chv\'atal closure, the $+$-cut closure of \cite{pokutta2010rank}, and the closure of Dantzig cuts derived from feasible bases. The failure for the Gomory fractional closure is perhaps surprising, as these cuts underlie a cutting-plane algorithm for pure integer programming that guarantees finite convergence, and the closure they define is a relaxation of the Chv\'atal closure, which does satisfy the hereditary property. The third contribution, Theorem~\ref{thm:corner-cut-hered}, is a sufficient condition for the hereditary property of closures built from valid inequalities for Gomory's corner polyhedron, where the cuts are indexed by the bases of the constraint matrix, so that no family of lattice-free convex sets common to the polyhedron and its faces is available. This condition establishes the hereditary property for the closure of Dantzig cuts derived from all bases, and it sheds light on the failures above, as the closures of Gomory fractional cuts and of Dantzig cuts derived from feasible bases satisfy one of its two requirements but not the other.

The bulk of prior work on the hereditary property for cutting-plane procedures is restricted to the rational polyhedral setting, with \cite{dunkel2013gomory} and \cite{dadush2014chvatal} being notable exceptions that treat irrational polytopes and compact convex sets, respectively, but only for the Chv\'atal closure. Moreover, the existing literature uniformly works with \emph{exposed} faces. While every face of a polyhedron is exposed, this need not be the case for general  convex sets. In contrast, our results address general convex sets and the broader notion of a face that need not be exposed. To the best of our knowledge, both extensions are new for the procedures considered in this work. We do, however, restrict to the polyhedral setting where the underlying structure asks it, as with Theorem~\ref{thm:corner-cut-hered} and the Sherali--Adams and Lasserre hierarchies.

The remainder of the paper is organized as follows. In Section~\ref{sec:intersection-cuts} we prove Theorem~\ref{thm:L-closure-hered} on the hereditary property for $\L$-closures. Section~\ref{sec:L-closure-applications} applies this result to derive the hereditary property for the split, lift-and-project, Lov\'asz--Schrijver, Sherali--Adams, and Lasserre closures, showing in each case that the procedure can be expressed as an $\L$-closure for a single family $\L$ shared by the underlying set and its faces. Section~\ref{sec:non-hered} presents four classical procedures for which the hereditary property fails. Finally, in Section~\ref{sec:corner-polyhedron} we prove Theorem~\ref{thm:corner-cut-hered}, a compatibility criterion for closures built from valid inequalities for Gomory's corner polyhedron, and apply it to establish the hereditary property for the closure of Dantzig cuts derived from all bases.

\subsection{Notation} \label{subsec:notation}

Throughout, $K \subseteq \R^n$ denotes a convex set. We let the mixed-integer feasible region be
\begin{align}
    S := \{x \in K : x_j \in \Z,\ j \in I\},
\end{align}
where $I \subseteq \{1,\ldots,n\}$ indexes the integer variables and $C := \{1,\ldots,n\} \setminus I$ the continuous ones. We write $\operatorname{int}(\cdot)$ to denote the interior of a set.

\section{Intersection Cuts and the \texorpdfstring{$\L$}{L}-Closure}
\label{sec:intersection-cuts}

We begin by introducing a geometric framework that encompasses a broad class of cutting-plane procedures. This abstraction is naturally captured by the notion of \emph{intersection cuts} (\cite{balas1971intersection}, \cite{averkov2011finite}), and yields our first main result.

\subsection{The \texorpdfstring{$\L$}{L}-Closure} \label{subsec:L-closure-defn}

Let $\mathcal{L}$ be a family of $n$-dimensional convex sets in $\R^n$, and let $K \subseteq \R^n$ be a convex set. Following \cite{averkov2012finitely}, define
\begin{align}
    R_L(K) &:= \operatorname{conv}(K \setminus \operatorname{int}(L)) \quad \text{for } L \in \L, \notag \\
    R_\L(K) &:= \bigcap_{L \in \L} R_L(K). \notag 
\end{align}
We refer to $R_L(K)$ as the \emph{$L$-reduction} of $K$, and to $R_{\mathcal{L}}(K)$ as its \emph{$\L$-closure}. 
Note that $R_L(K)$ is not a closed set in general, even when $K$ is a closed convex set and $L$ is a closed full-dimensional convex set (see \cite{averkov2011finite}, Figure 5). However, $R_L(K)$ is closed whenever $K$ is closed and $L$ is either a halfspace or $\operatorname{rec}(L)$ is a linear space (\cite{averkov2011finite}, Theorem 3.2). In particular, this holds when $L$ is a split. In such cases, $R_\mathcal{L}(K)$ is a closed convex set because it is the intersection of closed convex sets.

Following the notation of Section~\ref{subsec:notation} we consider the mixed-integer feasible region   $S = K \cap (\Z^{|I|} \times \R^{|C|})$. A set $L \subseteq \R^n$ is said to be \emph{lattice-free} if
\begin{align}
    \operatorname{int}(L) \cap (\Z^{|I|} \times \R^{|C|}) = \emptyset.
\end{align}
In the pure-integer case $C = \emptyset$ this is the classical integer-free condition $\operatorname{int}(L) \cap \Z^n = \emptyset$. If $L$ is lattice-free, then since $R_L(K) \subseteq K$ and $K \setminus \operatorname{int}(L) \subseteq R_L(K)$,
\begin{align}
    R_L(K) \cap (\Z^{|I|} \times \R^{|C|}) = K \cap (\Z^{|I|} \times \R^{|C|}) = S,
\end{align}
so any halfspace $H^+$ with $R_L(K) \subseteq H^+$ defines an inequality valid for $S$. If in addition $K \nsubseteq H^+$, the boundary of $H^+$ separates some point of $K$ from $R_L(K)$, and the resulting inequality is called an \emph{intersection cut} associated with $L$.

\subsection{The Hereditary Property for \texorpdfstring{$\L$}{L}-Closures} \label{subsec:L-closure-hered}

We now turn to the central result of this section, which states that any $\L$-closure satisfies the hereditary property, provided the same family $\L$ is used for a convex set and its faces. We begin with a basic and well-known property of faces of convex sets (see, \eg \cite{rockafellar1970convex}, Section 18). We include a short proof for completeness.

\begin{lemma} \label{lemma:face-decomposition}
    Let $F$ be a face of a convex set $K$. If $z = \sum_{i=1}^{r} \lambda_i p_i$ for some integer $r \geq 1$, where $p_i \in K$, $\lambda_i > 0$, $\sum_{i=1}^{r} \lambda_i = 1$, and $z \in F$, then $p_i \in F$ for all $i = 1,\ldots,r$.
\end{lemma}

\begin{proof}
    The proof is by induction on $r$. The case $r = 1$ is trivial, and the case $r = 2$ is the defining property of a face. For $r \geq 3$, set $\mu := 1 - \lambda_1 \in (0,1)$ and $q := \sum_{i=2}^{r} (\lambda_i/\mu)\, p_i$. Since $K$ is convex, $q \in K$, and $z = \lambda_1 p_1 + \mu q$ is a convex combination of two points of $K$ with $\lambda_1, \mu \in (0,1)$. Since $z \in F$, the definition of a face of a convex set implies that $p_1, q \in F$. Applying the induction hypothesis to $q = \sum_{i=2}^{r} (\lambda_i/\mu)\, p_i \in F$ yields $p_i \in F$ for $i = 2,\ldots,r$.
\end{proof}

\begin{theorem} \label{thm:L-closure-hered}
    Let $K \subseteq \R^n$ be a convex set, and let $\L$ be any family of full-dimensional convex sets in $\R^n$. Then, for any face $F$ of $K$, 
    \begin{align}
        R_\L(F) = R_\L(K) \cap F.
    \end{align}
\end{theorem}

\begin{proof}
    We first prove the inclusion $R_\L(F) \subseteq R_\L(K) \cap F$. Fix $L \in \L$. Since $F \subseteq K$, we have $F \setminus \operatorname{int}(L) \subseteq K \setminus \operatorname{int}(L)$, and the convexity of $F$ gives $\operatorname{conv}(F \setminus \operatorname{int}(L)) \subseteq F$. Hence
    \begin{align}
        \operatorname{conv}(F \setminus \operatorname{int}(L)) \subseteq F \cap \operatorname{conv}(K \setminus \operatorname{int}(L)) = F \cap R_L(K).
    \end{align}
    Intersecting over all $L \in \L$ yields $R_\L(F) \subseteq R_\L(K) \cap F$.

    Next, we prove the reverse inclusion. Let $x \in R_\L(K) \cap F$ and fix $L \in \L$. Because $R_L(K) = \operatorname{conv}(K \setminus \operatorname{int}(L))$ is a convex set, there exist finitely many points $p_1,\ldots,p_r \in K \setminus \operatorname{int}(L)$ and scalars $\lambda_1,\ldots,\lambda_r > 0$ with $\sum_{i=1}^{r} \lambda_i = 1$ such that $x = \sum_{i=1}^{r} \lambda_i p_i$. Since $x \in F$, Lemma~\ref{lemma:face-decomposition} gives $p_i \in F$ for all $i$, and consequently $p_i \in F \setminus \operatorname{int}(L)$. Therefore $x \in \operatorname{conv}(F \setminus \operatorname{int}(L)) = R_L(F)$. Since $L \in \L$ was arbitrary, $x \in R_\L(F)$.
\end{proof}

The only hypothesis is that the \emph{same} family $\L$ be used for both $K$ and its face $F$. The family itself is arbitrary, and the proof applies to every face $F$, without requiring exposedness. Consequently, to derive the hereditary property of a cutting-plane procedure from Theorem~\ref{thm:L-closure-hered}, it suffices to exhibit a single family $\L$ such that the procedure applied to $K$ yields $R_\L(K)$ and the procedure applied to $F$ yields $R_\L(F)$. A family chosen independently of the underlying set is automatically common to $K$ and $F$. By contrast, procedures such as the Chv\'atal-Gomory and Dantzig closures are intrinsically tied to the description of the set to which they are applied. When applied to a face $F$ of $K$, they generate cuts from the description of $F$, and no common family is apparent. For several such procedures the hereditary property fails outright (Section~\ref{sec:non-hered}), and Theorem~\ref{thm:L-closure-hered} then rules out any common family. The Gomory fractional cut closure and the Dantzig closure are discussed within the corner-polyhedron framework of Section~\ref{sec:corner-polyhedron}.

\section{The Hereditary Property of Cutting-Plane Procedures via \texorpdfstring{$\L$}{L}-Closures} \label{sec:L-closure-applications}

In this section, we apply Theorem~\ref{thm:L-closure-hered} to derive the hereditary property of several classical cutting-plane procedures. We consider the setup of Section~\ref{subsec:notation}. For each procedure we show that the closures of the underlying convex set and of its faces can be expressed as $\L$-closures for a common family $\L$, so that Theorem~\ref{thm:L-closure-hered} applies. For the split, lift-and-project, Lov\'asz--Schrijver, and Sherali--Adams procedures, the family does not depend on the underlying convex set, and is therefore automatically common. For the Lasserre hierarchy, the family is built from the inequalities describing the polyhedron, and showing that it also realizes the closure of each face requires an additional argument.

\subsection{Split Closure} \label{sec:split}

In this section, we discuss the split cuts of \cite{cook1990chvatal}, a special case of Balas' disjunctive cuts (\cite{balas1979disjunctive}). We show that the split closure is an $\L$-closure for a family $\L$ that does not depend on $K$, and is therefore common to $K$ and its faces. The hereditary property follows from Theorem~\ref{thm:L-closure-hered}.

A \emph{split} is a pair $(\pi, \pi_0) \in \Z^n \times \Z$ such that $\pi_j = 0$ for all $j \in C$.
Given $(\pi, \pi_0)$, define the disjunctive sets
\begin{align}
    \Pi_1 &:= K \cap \{x : \pi x \leq \pi_0\}, &
    \Pi_2 &:= K \cap \{x : \pi x \geq \pi_0 + 1\}, 
\end{align}
\begin{align}
   \mbox{ and } \;\;\;\; K^{(\pi, \pi_0)} := \operatorname{conv}(\Pi_1 \cup \Pi_2).
\end{align}
A \emph{split cut} is any linear inequality that is valid for $K^{(\pi, \pi_0)}$ for some split $(\pi, \pi_0)$ but is not valid for $K$.
The \emph{split closure} of $K$ is defined as 
\begin{align}
    K' := \bigcap_{(\pi, \pi_0)} K^{(\pi, \pi_0)},
\end{align}
where the intersection ranges over all splits.
Define $K^{(1)} := K'$ and $K^{(t)} := (K^{(t-1)})'$ for $t \geq 2$, the $t^{\text{th}}$ split closure of $K$.
If $K$ is a rational polyhedron, then so is $K'$ (\cite{cook1990chvatal}). If $K$ is a full dimensional, strictly convex, compact set, then $K'$ is defined by a finite number of splits, but is not necessarily polyhedral (\cite{dadush2011split}).


Define, for each split $(\pi, \pi_0)$, the lattice-free strip
\begin{align}
    L(\pi, \pi_0) := \{x \in \R^n : \pi_0 \leq \pi x \leq \pi_0 + 1\},
\end{align}
and let $\L_{\mathrm{split}} := \{L(\pi, \pi_0) : (\pi, \pi_0) \in \Z^n \times \Z, \ \pi_j = 0 \ \forall j \in C\}$. Note that $L(\pi, \pi_0)$ is full-dimensional and that the interior of $L(\pi, \pi_0)$ is $\{x : \pi_0 < \pi x < \pi_0 + 1\}$. Since the support of $\pi$ is contained in $I$, we have $\pi x \in \Z$ for every $x \in \Z^{|I|} \times \R^{|C|}$, so the interior is disjoint from $\Z^{|I|} \times \R^{|C|}$. Hence $L(\pi, \pi_0)$ is lattice-free in the mixed-integer sense.

For each split $(\pi, \pi_0)$, the $L$-reduction
\begin{align}
    R_{L(\pi,\pi_0)}(K) = \operatorname{conv}\big( (K \cap \{x : \pi x \leq \pi_0\}) \cup (K \cap \{x : \pi x \geq \pi_0 + 1\}) \big) = K^{(\pi, \pi_0)}
\end{align}
coincides with the split relaxation $K^{(\pi, \pi_0)}$. Intersecting over all splits gives
\begin{align}
    R_{\L_{\mathrm{split}}}(K) = K'.
\end{align}
Thus the split closure of $K$ is precisely the $\L_{\mathrm{split}}$-closure of $K$. Since $\L_{\mathrm{split}}$ depends only on the integer index set $I$, and not on $K$, Theorem~\ref{thm:L-closure-hered} applies and yields the following corollary.

\begin{corollary}[Split closure] \label{cor:split-hered}
    For any  convex set $K \subseteq \R^n$ and any face $F$ of $K$,
    \begin{align}
        F' = K' \cap F. \notag
    \end{align}
    By induction on $t$, $F^{(t)} = K^{(t)} \cap F$ for all $t \geq 1$.
\end{corollary}

\subsection{Lift-and-Project Closure} \label{sec:lift-project}

We next consider the lift-and-project cuts of \cite{balas1993lift}, which arise from the binary disjunction $x_i \in \{0,1\}$ for each integer-restricted variable $x_i, i \in I \subseteq \{ 1, \ldots, n \}$.

Suppose the integer variables indexed by $I$ are binary, so that
\begin{align}
    0 \leq x_i \leq 1 \quad \text{for all } i \in I.
\end{align}
The \emph{lift-and-project closure} of $K$ is defined as
\begin{align}
    \operatorname{LP}(K) := \bigcap_{i \in I}
    \operatorname{conv}\big((K \cap \{x_i = 0\}) \cup (K \cap \{x_i = 1\})\big).
\end{align}
Let $\operatorname{LP}^0(K) := K$ and define recursively $\operatorname{LP}^{k}(K) := \operatorname{LP}(\operatorname{LP}^{k-1}(K))$ for $k \geq 1$.
By Balas' sequential convexification theorem, $\operatorname{LP}^{|I|}(K) = \operatorname{conv}(S)$.


For each binary variable $x_i$, the disjunction $x_i \in \{0,1\}$ can be represented by the lattice-free set
\begin{align}
    L_i := \{x \in \R^n : 0 \leq x_i \leq 1\}.
\end{align}
Choosing $\L_{\mathrm{lp}} := \{L_i : i \in I\}$ and applying the reduction operator gives
\begin{align}
    R_{\L_{\mathrm{lp}}}(K) = \bigcap_{i \in I} \operatorname{conv}\big((K \cap \{x_i = 0\}) \cup (K \cap \{x_i = 1\})\big) = \operatorname{LP}(K),
\end{align}
which coincides exactly with the lift-and-project closure. Since $\L_{\mathrm{lp}}$ depends only on $I$ and not on $K$, Theorem~\ref{thm:L-closure-hered} applies.

\begin{corollary}[Lift-and-project closure] \label{cor:lp-hered}
    Let $K \subseteq \R^n$ be a convex set such that $0 \leq x_i \leq 1$ is valid for all $i \in I$, and let $F$ be a face of $K$. Then
    \begin{align}
        \operatorname{LP}(F) = \operatorname{LP}(K) \cap F,
    \end{align}
    and by induction $\operatorname{LP}^k(F) = \operatorname{LP}^k(K) \cap F$ for all $k \geq 1$.
\end{corollary}

\subsection{Lov\'asz--Schrijver Procedure} \label{sec:lovasz-schrijver}

We now turn to the procedure of \cite{lovasz1991cones}. This procedure operates in a lifted matrix space and is therefore not, at first sight, an $\L$-closure in the original variables. Nevertheless, after the lift, both the diagonal-equality strengthening and the positive semidefinite strengthening can be expressed as intersection cuts with respect to families that do not depend on $K$. Theorem~\ref{thm:L-closure-hered}, applied in the lifted space, then yields the hereditary property after projection.

We present the Lov\'asz--Schrijver procedure in the mixed-integer setting, where the integer-restricted variables indexed by $I$ satisfy $0 \leq x_i \leq 1$ and the continuous variables indexed by $C$ are unrestricted. Throughout this subsection, $K$ is a convex set with $K \subseteq \{x \in \R^n : 0 \leq x_i \leq 1,\ i \in I\}$, where $I \neq \emptyset$. The pure binary case is recovered when $C = \emptyset$ and $K \subseteq [0,1]^n$.

Define the \emph{homogenization cone} of $K$ as
\begin{align}
    \tilde K := \big\{ \lambda \begin{pmatrix} 1 \\ x \end{pmatrix} : x \in K,\ \lambda \geq 0 \big\} \subseteq \R^{n+1}, \label{def:hom-cone}
\end{align}
where the additional coordinate is indexed by $0$, and let $e^0,e^1,\ldots,e^n$ denote the unit vectors in $\R^{n+1}$. The procedure proceeds as follows:

\begin{enumerate}
    \item Let $\hat M(K)$ denote the set of symmetric $(n+1)\times(n+1)$ matrices $Y$ such that
    \begin{align}
        y_{00} = 1, \qquad Ye^i \in \tilde K, \qquad Ye^0 - Ye^i \in \tilde K, \qquad i \in I. \label{ls-cond-hat}
    \end{align}
    We can view the set $\hat M(K)$ as a subset of $\R^{\frac{(n+1)(n+2)}{2}}$ because of the symmetry $y_{ij} = y_{ji}$ for $i \not= j$. \eqref{ls-cond-hat} are imposed only for $i \in I$. To see why, note that for any $x \in S$ and the rank-one lift $Y = \binom{1}{x}(1 \ x)$,
    \begin{align}
        Ye^i = x_i\binom{1}{x}, \qquad Ye^0 - Ye^i = (1-x_i)\binom{1}{x}.
    \end{align}
    Both lie in the convex cone $\tilde K$ exactly when $x_i \geq 0$ and $1 - x_i \geq 0$, which is guaranteed only for $i \in I$. The matrix $Y$ retains all $(n+1)$ rows and columns. By symmetry, the constraints~\eqref{ls-cond-hat} hold for every entry $y_{ki}$ with $i \in I$ (and hence $y_{ik}$ as well), and the implicit constraint $Ye^0 \in \tilde K$, which is obtained by summing the two memberships and using that $\tilde K$ is a convex cone, holds for the $0$-th column. The only entries of $Y$ left fully unconstrained by~\eqref{ls-cond-hat} are those in the $C \times C$ block $\{y_{jk} : j,k \in C\}$.
    \item Strengthen $\hat M(K)$ by imposing the diagonal equalities $y_{ii} = y_{0i}$ for $i \in I$, and denote the resulting set by $M(K)$. These equalities encode $x_i^2 = x_i$ and are therefore imposed only for the integer-restricted variables.
    \item Let $M_+(K) := \{Y \in M(K) : Y \succeq 0\}$.
    \item Project $M_+(K)$ onto the $x$-space:
    \begin{align}
        N_+(K) := \big\{x \in \R^n : \begin{pmatrix} 1 \\ x \end{pmatrix} = Ye^0 \text{ for some } Y \in M_+(K)\big\}.
    \end{align}
\end{enumerate}

Observe that $N_+(K)$ is a convex set. Moreover, $N_+(K) \subseteq K$, because for any $i \in I$ and $Y \in M_+(K)$ with $Ye^0 = \binom{1}{x}$, we have $Ye^0 = Ye^i + (Ye^0 - Ye^i) \in \tilde K$, so $\binom{1}{x} \in \tilde K$ and hence $x \in K$. In particular, $N_+(K)$ inherits the bounds $0 \leq x_i \leq 1$ for $i \in I$. The procedure can therefore be iterated. We define $N_+^0(K) := K$ and $N_+^t(K) := N_+(N_+^{t-1}(K))$ for $t \geq 1$.

When $K =  \{x \in \R^n : 0 \leq x_i \leq 1,\ i \in I,\ a^\ell x \leq b_\ell,\ \ell = 1,\ldots,m\}$ is a polyhedron, the constraint $Ye^i \in \tilde K$ for $i \in I$ is equivalent to the linearized inequalities $b_\ell y_{0i} - \sum_k a^\ell_k y_{ki} \geq 0$ obtained from $x_i(b_\ell - a^\ell x) \geq 0$, and $Ye^0 - Ye^i \in \tilde K$ corresponds to the linearization of $(1-x_i)(b_\ell - a^\ell x) \geq 0$. These products are formed only for the integer-restricted variables $x_i$, $i \in I$, for which $x_i \geq 0$ and $1 - x_i \geq 0$ are valid.

\subsubsection{The Lov\'asz--Schrijver Strengthening as an \texorpdfstring{$\L$}{L}-Closure} \label{subsec:ls-as-intersection}

We first show that the strengthening obtained by enforcing $y_{0i} = y_{ii}$ in Step 2 of the Lov\'asz--Schrijver procedure arises as intersection cuts.

For a fixed $i \in I$, define the set $L_i := \{Y \in \R^{\frac{(n+1)(n+2)}{2}} : 0 \leq y_{0i} - y_{ii} \leq 1\}$.
Notice that this set is a split, which is closed, convex, and lattice-free.

\begin{lemma} \label{lemma:ls-first-strengthen}
    For each $i \in I$,
    \begin{align}
        \operatorname{conv}\bigl(\hat M(K) \setminus \operatorname{int}(L_i)\bigr)
        = \hat M(K) \cap \{ y_{ii} = y_{0i}\}.
    \end{align}
\end{lemma}

\begin{proof}
    By the definition of $L_i$,
    \begin{align}
        \hat M(K) \setminus \operatorname{int}(L_i)
        &= \bigl(\hat M(K) \cap \{y_{0i} - y_{ii} \leq 0\}\bigr)
        \cup \bigl(\hat M(K) \cap \{y_{0i} - y_{ii} \geq 1\}\bigr). \label{eq:ls-disjunction-matrix}
    \end{align}
    We claim that the McCormick-type inequalities
    \begin{align}
        y_{ii} \leq y_{0i}, \qquad 2y_{0i} - 1 \leq y_{ii} \label{ineq:mccormick-matrix}
    \end{align}
    are valid for every $Y \in \hat M(K)$. Indeed, since $Ye^i \in \tilde K$, there exist $\lambda \geq 0$ and $x \in K$ with $Ye^i = \lambda \binom{1}{x}$. The $0$-th coordinate gives $\lambda = y_{0i}$, and the $i$-th coordinate gives $y_{ii} = \lambda x_i = y_{0i}\, x_i$. Since $i \in I$ and hence $0 \leq x_i \leq 1$, this yields $0 \leq y_{ii} \leq y_{0i}$. Similarly, $Ye^0 - Ye^i \in \tilde K$ gives $Ye^0 - Ye^i = \mu\binom{1}{x'}$ for some $\mu \geq 0$ and $x' \in K$. The $0$-th coordinate gives $\mu = 1 - y_{0i}$, and the $i$-th coordinate gives $y_{0i} - y_{ii} = (1-y_{0i})x'_i \leq 1 - y_{0i}$, which rearranges to $2y_{0i} - 1 \leq y_{ii}$.

    We observe from the first relation in \eqref{ineq:mccormick-matrix} that the first disjunction in \eqref{eq:ls-disjunction-matrix} reduces to $\hat M(K) \cap \{y_{ii} = y_{0i}\}$.
    For the second disjunction, combining $y_{0i} - y_{ii} \geq 1$ with $2y_{0i} - 1 \leq y_{ii}$ in \eqref{ineq:mccormick-matrix} yields $y_{ii} \leq -1$,
    which contradicts $y_{ii} \geq 0$. Hence the second disjunction is empty, and
    \begin{align}
        \hat M(K) \setminus \operatorname{int}(L_i) &= \hat M(K) \cap \{y_{ii} = y_{0i}\}.
    \end{align}
    Since both sets $\hat M(K)$ and $\{y_{ii} = y_{0i}\}$ are convex, so is their intersection. Taking the convex hull does not enlarge the set:
    \begin{align}
        \operatorname{conv}\bigl(\hat M(K) \setminus \operatorname{int}(L_i)\bigr) = \hat M(K) \cap \{y_{ii} = y_{0i}\}.
    \end{align}
\end{proof}

Thus, the substitution $y_{ii} = y_{0i}$ is precisely the result of applying an intersection cut with respect to $L_i$. Taking all $i \in I$, let $\L := \{L_i : i \in I\}$ be a family of such closed, convex, lattice-free sets.

\begin{lemma} \label{lemma:ls-family} 
    $R_\L(\hat M(K)) = M(K)$.
\end{lemma}
\begin{proof}
    \begin{align}
        R_\L(\hat M(K)) &= \bigcap_{L_i \in \L} \operatorname{conv}\big(\hat M(K) \setminus \operatorname{int}(L_i)\big) \\
        &= \bigcap_{i \in I} \big(\hat M(K) \cap \{y_{ii} = y_{0i}\}\big)  \\
        &= \hat M(K) \cap \{y_{ii} = y_{0i}, \ \forall i \in I\} \\
        &=: M(K),
    \end{align}
    where the second equality follows from Lemma \ref{lemma:ls-first-strengthen}.
\end{proof}

We next show that the positive semidefinite strengthening in Step 3 of the Lov\'asz-Schrijver procedure can also be interpreted as intersection cuts.

Consider the cone of symmetric positive semidefinite matrices $\mathbb{S}^{n+1}_+ = \{Y \in \R^{\frac{(n+1)(n+2)}{2}} : Y \succeq 0\}$. It is well known that $\mathbb{S}^{n+1}_+$ is a closed convex cone and that all rank-one matrices lie on its boundary (see, for example, \cite{ben2001lectures}, Section 4.1, or \cite{boyd2004convex}, Section 2.4.1, Example 2.15).

Let $\mathcal{R} := \{xx^\top : x \in \R^{n+1}\setminus\{0\}\}$ denote the set of rank-one symmetric $(n+1) \times (n+1)$ matrices. Note that every rank-one matrix $xx^\top$ is positive semidefinite, so $\mathcal{R} \subseteq \mathbb{S}^{n+1}_+$.
A convex set $L \subseteq \R^{\frac{(n+1)(n+2)}{2}}$ is said to be \emph{$\mathcal{R}$-free} if $\operatorname{int}(L) \cap \mathcal{R} = \emptyset$.

Let $\L'$ be the family of all closed halfspaces $L \subseteq \R^{\frac{(n+1)(n+2)}{2}}$ whose interior is disjoint from $\mathbb{S}^{n+1}_+$:
\begin{align}
    \L' := \{L \subseteq \R^{\frac{(n+1)(n+2)}{2}} : L \text{ is a closed halfspace with } \operatorname{int}(L) \cap \mathbb{S}^{n+1}_+ = \emptyset\}.
\end{align}
Each $L \in \L'$ is $\mathcal{R}$-free since $\mathcal{R} \subseteq \mathbb{S}^{n+1}_+$. Equivalently, every $L \in \L'$ has the form $L = \{Z : \alpha^\top Z \geq \beta\}$, where $H^+_L := \{Z : \alpha^\top Z \leq \beta\}$ is a closed halfspace containing $\mathbb{S}^{n+1}_+$. Since $\mathbb{S}^{n+1}_+$ is a closed convex set, it equals the intersection of all such supporting halfspaces:
\begin{align}
    \mathbb{S}^{n+1}_+ = \bigcap_{L \in \L'} H^+_L. \label{eq:psd-halfspace-intersection}
\end{align}

\begin{lemma} \label{lemma:ls-psd-strengthen}
    For any convex set $C \subseteq \R^{\frac{(n+1)(n+2)}{2}}$,
    \begin{align}
        R_{\L'}(C) = C \cap \mathbb{S}^{n+1}_+. \notag
    \end{align}
    In particular, $R_{\L'}(M(K)) = M_+(K)$.
\end{lemma}

\begin{proof}
    For each $L \in \L'$, $C \setminus \operatorname{int}(L) = C \cap H^+_L$ is the intersection of two convex sets, and is hence convex. Thus
    \begin{align}
        \operatorname{conv}(C \setminus \operatorname{int}(L)) = C \cap H^+_L.
    \end{align}
    Taking the intersection over all $L \in \L'$ and using \eqref{eq:psd-halfspace-intersection},
    \begin{align}
        R_{\L'}(C) = \bigcap_{L \in \L'}(C \cap H^+_L) = C \cap \bigcap_{L \in \L'} H^+_L = C \cap \mathbb{S}^{n+1}_+.
    \end{align}
    Setting $C = M(K)$ and recalling $M_+(K) = M(K) \cap \mathbb{S}^{n+1}_+$ proves the second statement.
\end{proof}

\medskip

Both families $\L$ and $\L'$ live in the same lifted space $\R^{\frac{(n+1)(n+2)}{2}}$. Thus, we may define a unified family of lattice-free convex sets $\L^* = \L \cup \L'$, and the entire Lov\'asz--Schrijver strengthening can be interpreted as the $\L^*$-closure of $\hat M(K)$.

\begin{corollary} \label{cor3}
    Let $\L^* := \L \cup \L'$. Then
    \begin{align}
        M_+(K) = R_{\L^*}(\hat M(K)).
    \end{align}
\end{corollary}

\begin{proof}
    Since $R_{\L \cup \L'}(\hat M(K)) = R_\L(\hat M(K)) \cap R_{\L'}(\hat M(K))$, Lemma~\ref{lemma:ls-family} gives $R_\L(\hat M(K)) = M(K)$, and Lemma~\ref{lemma:ls-psd-strengthen} applied with $C = \hat M(K)$ gives $R_{\L'}(\hat M(K)) = \hat M(K) \cap \mathbb{S}^{n+1}_+$. Hence
    \begin{align}
        R_{\L^*}(\hat M(K)) = M(K) \cap \hat M(K) \cap \mathbb{S}^{n+1}_+ = M(K) \cap \mathbb{S}^{n+1}_+ = M_+(K),
    \end{align}
    where the second equality uses $M(K) \subseteq \hat M(K)$.
\end{proof}

In this way, the Lov\'asz--Schrijver procedure admits a natural geometric interpretation. It strengthens the lifted relaxation $\hat M(K)$ through the successive application of intersection cuts derived from a family of lattice-free convex sets that enforce both the diagonal equalities and the positive semidefinite constraint.

\subsubsection{The Hereditary Property} \label{subsec:ls-hered}

The argument now proceeds in two steps. We first apply Theorem~\ref{thm:L-closure-hered} in the lifted space, and then project to the original space. The key technical point is that the lift of a face of $K$ is itself a face of $\hat M(K)$.

\begin{corollary}[Lov\'asz--Schrijver procedure] \label{cor:ls-hered}
    Let $K \subseteq \{x \in \R^n : 0 \leq x_i \leq 1,\ i \in I\}$ be a convex set, and let $F$ be a face of $K$. Then
    \begin{align}
        N_+(F) = N_+(K) \cap F.
    \end{align}
\end{corollary}

\begin{proof}
    We first verify that $\hat M(F)$ is a face of $\hat M(K)$. We begin by showing that the homogenization cone $\tilde F$ is a face of $\tilde K$. Let $u, v \in \tilde K$ and $\lambda \in (0,1)$ with $\lambda u + (1-\lambda) v \in \tilde F$. Write $u = \alpha \binom{1}{x^1}$ and $v = \beta \binom{1}{x^2}$ for some $\alpha,\beta \geq 0$ and $x^1, x^2 \in K$, and write $\lambda u + (1-\lambda) v = \mu \binom{1}{y}$ for some $\mu \geq 0$ and $y \in F$. Comparing $0$-th coordinates gives $\mu = \lambda \alpha + (1-\lambda)\beta$. If $\mu = 0$, then $\alpha = \beta = 0$ and $u = v = 0 \in \tilde F$. If $\mu > 0$, then $y = \frac{\lambda \alpha}{\mu} x^1 + \frac{(1-\lambda)\beta}{\mu} x^2 \in F$ is a convex combination of $x^1, x^2 \in K$ with positive coefficients, whenever $\alpha,\beta > 0$. By the definition of a face of a convex set, $x^1, x^2 \in F$, so $u, v \in \tilde F$. If, say, $\alpha = 0$, then $u = 0 \in \tilde F$ trivially, and the same argument applied to $v$ gives $v \in \tilde F$. Hence $\tilde F$ is a face of $\tilde K$.

    Now suppose $Y, Z \in \hat M(K)$ and $\lambda \in (0,1)$ are such that $\lambda Y + (1-\lambda) Z \in \hat M(F)$. For each $i \in I$, $Ye^i, Ze^i \in \tilde K$ and $\lambda(Ye^i) + (1-\lambda)(Ze^i) \in \tilde F$, so the definition of a face gives $Ye^i, Ze^i \in \tilde F$. The same argument applied to $Ye^0 - Ye^i$ and $Ze^0 - Ze^i$ shows $Ye^0 - Ye^i, Ze^0 - Ze^i \in \tilde F$. Since the memberships~\eqref{ls-cond-hat} defining $\hat M$ involve only the indices $i \in I$, it follows that $Y, Z \in \hat M(F)$, and $\hat M(F)$ is a face of $\hat M(K)$.

    Let $\L^* = \L \cup \L'$ as defined in Section~\ref{subsec:ls-as-intersection}, where $\L$ encodes the diagonal equalities and $\L'$ encodes the positive semidefiniteness of $Y$. Both $\L$ and $\L'$ depend only on the ambient lifted dimension and not on $K$ or $F$. Applying Theorem~\ref{thm:L-closure-hered} to the face $\hat M(F)$ of $\hat M(K)$ gives
    \begin{align}
        M_+(F) = R_{\L^*}(\hat M(F)) = R_{\L^*}(\hat M(K)) \cap \hat M(F) = M_+(K) \cap \hat M(F). \label{eq:M-plus-F-lifted}
    \end{align}
    Above, the first and last equations follow from Corollary~\ref{cor3} and the second equation from Theorem~\ref{thm:L-closure-hered}.

    It remains to project \eqref{eq:M-plus-F-lifted} onto the $x$-space. For the inclusion $N_+(F) \subseteq N_+(K) \cap F$, any $Y \in M_+(F)$ with $Ye^0 = \binom{1}{x}$ lies in $M_+(K)$ (so $x \in N_+(K)$) and in $\hat M(F)$. Fixing $i \in I$, $Ye^0 = Ye^i + (Ye^0 - Ye^i) \in \tilde F$ since $\tilde F$ is a convex cone, which forces $x \in F$. 
    
    For the reverse inclusion, let $x \in N_+(K) \cap F$ and pick $Y \in M_+(K)$ with $Ye^0 = \binom{1}{x}$. We show $Y \in \hat M(F)$. Fix $i \in I$ and write $Ye^i = \alpha\binom{1}{u}$ and $Ye^0 - Ye^i = \beta\binom{1}{v}$ with $\alpha, \beta \geq 0$ and $u, v \in K$ when the corresponding scalar is positive. Adding the two equations, we get $\binom{\alpha + \beta}{\alpha u + \beta v} = Ye^0 = \binom{1}{x}$.
    The $0$-th coordinates give $\alpha + \beta = 1$ and the remaining coordinates give $\alpha u + \beta v = x$. If $\alpha, \beta > 0$, then $x$ is a strict convex combination of $u, v \in K$, so by the definition of a face $u, v \in F$, and hence $Ye^i, Ye^0 - Ye^i \in \tilde F$. If $\alpha = 0$, then $Ye^i = 0 \in \tilde F$ and $Ye^0 - Ye^i = \binom{1}{x} \in \tilde F$. The case $\beta = 0$ is symmetric. Since the memberships~\eqref{ls-cond-hat} defining $\hat M(F)$ involve only the indices $i \in I$, it follows that $Y \in \hat M(F)$.
    
    It follows that $Y \in M_+(F)$ by \eqref{eq:M-plus-F-lifted} and $x \in N_+(F)$.
\end{proof}

By induction, applying the Corollary~\ref{cor:ls-hered} to the iterated convex sets $N_+^t(K)$ yields the hereditary property for every level of the Lov\'asz--Schrijver hierarchy, \ie $N_+^t(F) = N_+^t(K) \cap F$ for all $t \geq 1$. The iteration is well defined since $N_+(K) \subseteq K$ preserves the bounds $0 \leq x_i \leq 1$ for $i \in I$.

\subsection{The Sherali--Adams Hierarchy} \label{sec:sherali-adams}

We now consider the hierarchy of \cite{sherali1990hierarchy}. 

We work with a rational polyhedron $P \subseteq \R^n$ described by inequalities $b_\ell - {a^\ell} x \geq 0$ for $\ell = 1,\dots,m$, with integer index set $I$ and continuous index set $C := \{1,\ldots,n\} \setminus I$. We assume that $P$ includes the bound constraints $0 \leq x_i \leq 1$ for every $i \in I$, so that the integer-restricted variables are binary, and we write $S := P \cap (\{0,1\}^{|I|} \times \R^{|C|})$.

As with the Lov\'asz--Schrijver procedure, the Sherali--Adams hierarchy operates in a lifted space, and the strengthening step can be expressed as an intersection-cut closure with a family that does not depend on the underlying polyhedron. Throughout, we adopt the equivalent reformulation of the level-$t$ Sherali--Adams hierarchy given by \cite{bodur2017new}, which is formulated directly in the mixed 0-1 setting.

 Fix an integer $t$ with $1 \leq t \leq |I|$. A \emph{pair of order $r$} is a pair of disjoint subsets $I_1, I_2 \subseteq I$ with $|I_1|+|I_2| = r$. For each pair $(I_1, I_2)$ of order $r \leq t$, define
\begin{align}
    \F_r(I_1, I_2) := \prod_{i \in I_1} x_i \prod_{j \in I_2} (1 - x_j),
    \qquad \F_0(\emptyset,\emptyset) := 1.
\end{align}

The level-$t$ Sherali--Adams procedure proceeds as follows:
\begin{enumerate}
    \item For each pair $(I_1, I_2)$ of order $t$ and each $\ell \in \{1,\ldots,m\}$, form the polynomial inequality
    \begin{align}
        \F_t(I_1, I_2)\,(b_\ell - {a^\ell} x) \geq 0. \label{sa-lifting-ineq}
    \end{align}
    The multiplier $\F_t(I_1, I_2)$ ranges only over indices in $I$, since only these admit the bounds $0 \leq x_i \leq 1$ needed for $\F_t(I_1, I_2) \geq 0$ to be valid.

    \item Let $\M^t$ denote the collection of ordered tuples $\S$ of indices from $\{1,\dots,n\}$ satisfying
    \begin{enumerate*}[label=(\roman*)]
        \item $|\S| \leq t+1$,
        \item the entries are in nondecreasing order of indices,
        \item at most one index lies outside $I$, and
        \item at most one index from $I$ is repeated.
    \end{enumerate*}
    The last two conditions reflect that a monomial $\prod_{r=1}^{|\S|} x_{\S_r}$ arising from expanding \eqref{sa-lifting-ineq} carries at most one continuous factor (from $b_\ell - {a^\ell} x$), and only the binary factors $x_i^2 = x_i$, $i \in I$, can produce repetitions. Each inequality \eqref{sa-lifting-ineq} can be written as $\beta - \sum_{\S \in \Omega(I_1, I_2,\ell)} \alpha_\S \prod_{r=1}^{|\S|} x_{\S_r} \geq 0$ for some scalars $\beta, \alpha_\S \in \R$ and $\Omega(I_1, I_2, \ell) \subseteq \M^t$. Substituting $y_\S$ for the monomial $\prod_{r=1}^{|\S|} x_{\S_r}$ yields
    \begin{align}
        \beta - \sum_{\S \in \Omega(I_1, I_2, \ell)} \alpha_\S y_\S \geq 0. \label{sa-lifting-linear-ineq}
    \end{align}
    Denote by $\widehat{SA}_t(P) \subseteq \R^{|\M^t|}$ the polyhedron defined by all inequalities \eqref{sa-lifting-linear-ineq} taken over pairs $(I_1, I_2)$ of order $t$ and $\ell \in \{1,\ldots,m\}$, identifying the original variables $x_1,\ldots,x_n$ with $y_{(1)},\ldots,y_{(n)}$.

    \item For $\S \in \M^t$, let $[\S]$ denote the tuple obtained by deleting the repeated index (if any) while preserving order. For example, if $\S_1 = (1,1,2)$ and $\S_2 = (1,2,3)$, then $[\S_1] = (1,2)$ and $[\S_2] = (1,2,3)$. Impose the equalities $y_\S = y_{[\S]}$ for all $\S \in \M^t$, which are valid for $\{0,1\}^{|I|} \times \R^{|C|}$ since $x_i^2 = x_i$ for $i \in I$ and no repetitions occur outside $I$. Denote the resulting polyhedron by $\widetilde{SA}_t(P)$.

    \item Project $\widetilde{SA}_t(P)$ onto $\R^n$:
    \begin{align}
        SA_t(P) := \{x \in \R^n : \exists\, y \in \widetilde{SA}_t(P) \text{ with } x_j = y_{(j)},\ j = 1,\ldots,n\}.
    \end{align}
\end{enumerate}

A useful structural consequence of the Sherali--Adams construction is the following.

\begin{lemma}[\cite{sherali1990hierarchy}, Lemma 1] \label{lemma:sa-order-r}
    Let $0 \leq r \leq t$. For every pair $(I_1, I_2)$ of order $r$, the polynomial $\F_r(I_1, I_2)$ is a nonnegative combination of order-$t$ polynomials $\F_t(I_1', I_2')$.
\end{lemma}
Consequently, any polynomial inequality of the form $\F_r(I_1, I_2)(b_\ell - {a^\ell} x) \geq 0$ 
has a linearization valid for $\widehat{SA}_t(P)$.

\subsubsection{The Sherali--Adams Hierarchy as an \texorpdfstring{$\L$}{L}-Closure} \label{subsec:sa-as-L}

We show that the strengthening equalities $y_\S = y_{[\S]}$ arise as intersection cuts in the lifted space, with respect to a family of full-dimensional lattice-free sets. The key ingredient is the following lemma, which records four valid inequalities established by \cite{bodur2017new}. We refer the reader to that proof for the derivations from the underlying polynomial liftings.

\begin{lemma}[{from the proof of \cite{bodur2017new}, Lemma 3.9}] \label{lemma:bodur-3.9}
    Let $\S \in \M^t$ have a repeated index $i \in I$. The following inequalities are valid for $\widehat{SA}_t(P)$:
    \begin{align}
        0 \leq y_\S \leq y_{[\S]} \leq y_{(i)}, \qquad 1 - y_{(i)} - y_{[\S]} + y_\S \geq 0. \notag
    \end{align}
\end{lemma}

For each $\S \in \M^t$ with a repeated index, define the closed convex set
\begin{align}
    L_\S := \{y \in \R^{|\M^t|} : 0 \leq y_{[\S]} - y_{\S} \leq 1\}.
\end{align}
The set $L_\S$ is lattice-free in the mixed 0-1 sense. For every $\bar x \in \{0,1\}^{|I|} \times \R^{|C|}$, the lift $\bar y \in \R^{|\M^t|}$  satisfies $\bar y_\S = \bar y_{[\S]}$, since $\bar x_i^2 = \bar x_i$ for $i \in I$ and only such indices repeat in $\S$. Hence $\bar y_{[\S]} - \bar y_\S = 0 \notin (0,1)$, so the lift never lies in $\operatorname{int}(L_\S)$.

\begin{lemma} \label{lemma:sa-strengthen}
    For any $\S \in \M^t$ with a repeated index,
    \begin{align}
        \operatorname{conv}\bigl(\widehat{SA}_t(P) \setminus \operatorname{int}(L_\S)\bigr)
        = \widehat{SA}_t(P) \cap \{y_\S = y_{[\S]}\}. \notag
    \end{align}
\end{lemma}

\begin{proof}
    Let $i \in I$ be the repeated index of $\S$. By definition of $L_\S$,
    \begin{align}
        \widehat{SA}_t(P) \setminus \operatorname{int}(L_\S)
        = \bigl(\widehat{SA}_t(P) \cap \{y_{[\S]} - y_\S \leq 0\}\bigr) \cup \bigl(\widehat{SA}_t(P) \cap \{y_{[\S]} - y_\S \geq 1\}\bigr). \label{sa-disjunction}
    \end{align}
    By Lemma~\ref{lemma:bodur-3.9}, $y_\S \leq y_{[\S]}$ is valid for $\widehat{SA}_t(P)$, so the first set reduces to $\widehat{SA}_t(P) \cap \{y_\S = y_{[\S]}\}$. For the second set, combining $y_{[\S]} - y_\S \geq 1$ with $1 - y_{(i)} - y_{[\S]} + y_\S \geq 0$ from Lemma~\ref{lemma:bodur-3.9} yields $y_{(i)} \leq 0$, and the chain $0 \leq y_\S \leq y_{[\S]} \leq y_{(i)} \leq 0$ forces $y_{[\S]} - y_\S = 0$, contradicting $y_{[\S]} - y_\S \geq 1$. The second set is therefore empty, and since $\widehat{SA}_t(P) \cap \{y_\S = y_{[\S]}\}$ is convex, taking the convex hull leaves it unchanged.
\end{proof}

Let $\L_{\mathrm{SA}} := \{L_\S : \S \in \M^t,\ \S \text{ has a repeated index}\}$. The lifted Sherali--Adams polyhedron is recovered as the $\L_{\mathrm{SA}}$-closure of $\widehat{SA}_t(P)$.

\begin{lemma} \label{lemma:sa-as-L}
    $R_{\L_{\mathrm{SA}}}(\widehat{SA}_t(P)) = \widetilde{SA}_t(P)$.
\end{lemma}

\begin{proof}
    By Lemma~\ref{lemma:sa-strengthen}, intersecting $\widehat{SA}_t(P)$ with the $L_\S$-reductions over all $\S \in \M^t$ with a repeated index imposes exactly the equalities $y_\S = y_{[\S]}$. For $\S$ without a repeated index, $[\S] = \S$ and the equality is trivial. Hence $R_{\L_{\mathrm{SA}}}(\widehat{SA}_t(P)) = \widehat{SA}_t(P) \cap \{y_\S = y_{[\S]} \text{ for all } \S \in \M^t\} = \widetilde{SA}_t(P)$.
\end{proof}

\subsubsection{The Hereditary Property} \label{subsec:sa-hered}

As with the Lov\'asz--Schrijver procedure, the hereditary property follows by applying Theorem~\ref{thm:L-closure-hered} in the lifted space and then projecting.

\begin{corollary}[Sherali--Adams hierarchy] \label{cor:sa-hered}
    Let $P \subseteq \R^n$ be a rational polyhedron with $0 \leq x_i \leq 1$ valid for all $i \in I$, let $F$ be a face of $P$, and fix $1 \leq t \leq |I|$. Then
    \begin{align}
        SA_t(F) = SA_t(P) \cap F. \notag
    \end{align}
\end{corollary}

\begin{proof}
    Write $F = P \cap \{x : cx = \delta\}$, where $cx \leq \delta$ is valid for $P$, so $F$ is described by the inequalities defining $P$ together with $cx \leq \delta$ and $-cx \leq -\delta$. For each pair $(I_1, I_2)$ of order $r \leq t$ (with $I_1, I_2 \subseteq I$), multiplying these last two inequalities by $\F_r(I_1, I_2)$ yields the polynomial inequalities $(\delta - cx)\F_r(I_1, I_2) \geq 0$ and $(cx - \delta)\F_r(I_1, I_2) \geq 0$. Each monomial appearing after expansion lies in $\M^t$. Indeed, $\F_r(I_1, I_2)$ contributes $r$ distinct indices from $I$ (since $I_1, I_2 \subseteq I$ are disjoint), and the factor $cx - \delta$ contributes at most one index from $\{1,\ldots,n\}$ on top. The resulting tuple has length at most $r+1 \leq t+1$, with at most one index outside $I$ and at most one repeated index in $I$, fulfilling the conditions on $\M^t$. Let $\mathcal{G}_{I_1, I_2}(y)$ denote the linearization of $(\delta - cx)\F_r(I_1, I_2)$. In particular, $\mathcal{G}_{\emptyset,\emptyset}(y) = \delta - \sum_i c_i y_{(i)}$. Since $cx \leq \delta$ is valid for $P$, Farkas' lemma expresses $\delta - cx$ as a nonnegative combination of the $b_\ell - {a^\ell} x$, and Lemma~\ref{lemma:sa-order-r} reduces the resulting products to products of order $t$. The linearizations of $(\delta - cx)\F_r(I_1, I_2) \geq 0$ for all pairs of order $r \leq t$ are therefore valid for $\widehat{SA}_t(P)$. Adding the reverse inequalities $-\mathcal{G}_{I_1, I_2}(y) \geq 0$ for pairs of order exactly $t$ defines $\widehat{SA}_t(F)$. The same reverse inequalities for pairs of order $r < t$ follow from those for order $t$. By Lemma~\ref{lemma:sa-order-r}, $\F_r(I_1, I_2)(cx-\delta) \geq 0$ is a nonnegative combination of the inequalities $\F_t(I_1', I_2')(cx-\delta) \geq 0$, and linearizing preserves this combination. Hence
    \begin{align}
        \widehat{SA}_t(F) = \widehat{SA}_t(P) \cap \bigl\{y : \mathcal{G}_{I_1, I_2}(y) = 0 \text{ for all pairs } (I_1, I_2) \text{ of order } \leq t\bigr\}, \label{eq:hat-SA-F}
    \end{align}
    which exhibits $\widehat{SA}_t(F)$ as a face of $\widehat{SA}_t(P)$.

    The family $\L_{\mathrm{SA}}$ from Section~\ref{subsec:sa-as-L} depends only on $\M^t$, not on $P$. Theorem~\ref{thm:L-closure-hered} applied in the lifted space, together with Lemma~\ref{lemma:sa-as-L}, yields
    \begin{align}
        \widetilde{SA}_t(F) = R_{\L_{\mathrm{SA}}}(\widehat{SA}_t(F)) = R_{\L_{\mathrm{SA}}}(\widehat{SA}_t(P)) \cap \widehat{SA}_t(F) = \widetilde{SA}_t(P) \cap \widehat{SA}_t(F). \label{eq:Rt-F-lifted}
    \end{align}
    It remains to project \eqref{eq:Rt-F-lifted} onto the $x$-space, where the projection sends $\bar y$ to $(\bar y_{(1)},\ldots,\bar y_{(n)})$.

    For $SA_t(F) \subseteq SA_t(P) \cap F$, any $\bar y \in \widetilde{SA}_t(F)$ lies in $\widetilde{SA}_t(P)$ (so its projection $\bar x \in SA_t(P)$) and in $\widehat{SA}_t(F)$, which by \eqref{eq:hat-SA-F} gives $\mathcal{G}_{\emptyset,\emptyset}(\bar y) = \delta - c\bar x = 0$, so $\bar x \in F$.

    For the reverse, let $\bar x \in SA_t(P) \cap F$ and pick $\bar y \in \widetilde{SA}_t(P)$ projecting to $\bar x$. By \eqref{eq:hat-SA-F} and \eqref{eq:Rt-F-lifted}, it suffices to show $\mathcal{G}_{I_1, I_2}(\bar y) = 0$ for every pair $(I_1, I_2)$ of order $r \leq t$. Set $T := I_1 \cup I_2 \subseteq I$. The polynomial identity $\sum_{I_1' \subseteq T} \F_r(I_1', T \setminus I_1') = \prod_{i \in T}(x_i + (1-x_i)) = 1$, multiplied by $\delta - cx$ and linearized, gives
    \begin{align}
        \sum_{I_1' \subseteq T} \mathcal{G}_{I_1', T \setminus I_1'}(\bar y) = \delta - c\bar x = 0,
    \end{align}
    using $\bar x \in F$. Each term is nonnegative since $\bar y \in \widehat{SA}_t(P)$ and $(I_1', T \setminus I_1')$ is a pair with $I_1', T \setminus I_1' \subseteq T \subseteq I$ of order $r \leq t$, hence $\mathcal{G}_{I_1, I_2}(\bar y) = 0$.
\end{proof}

\subsection{The Lasserre Hierarchy} \label{sec:lasserre}

We now turn to the hierarchy of \cite{lasserre2001explicit}, following the presentation of \cite{laurent2003comparison}. We work with a polyhedron $P \subseteq [0,1]^n$ described by linear inequalities in the pure binary setting, and show that the Lasserre relaxation can be expressed as an $\L$-closure. Unlike the preceding procedures, the family $\L$ is built from the inequalities describing $P$. The key step in establishing the hereditary property is therefore to show that this same family also realizes the Lasserre relaxation of each face of $P$ (Lemma~\ref{lemma:lasserre-face-redundancy}), after which Theorem~\ref{thm:L-closure-hered} applies.

Let $V := \{1,\ldots,n\}$ and let $\P_t(V)$ denote the family of subsets of $V$ of cardinality at most $t$. Consider a polyhedron
\begin{align}
    P := \{x \in [0,1]^n : g_\ell(x) \geq 0, \ \ell = 1,\ldots,m\},
\end{align}
where each $g_\ell(x) = b_\ell - {a^\ell} x$ is a linear function. We write $S := P \cap \{0,1\}^n$.

Given a vector $y \in \R^{\P(V)}$ with components $y_J$ for $J \subseteq V$, define the \emph{moment matrix} $M_t(y)$ as the symmetric matrix indexed by $\P_t(V)$ with entry
\begin{align}
    M_t(y)_{J, J'} := y_{J \cup J'}, \qquad J, J' \in \P_t(V). \label{def:moment-matrix}
\end{align}
For any $\bar x \in \{0,1\}^n$, the vector $\bar y \in \R^{\P(V)}$ with $\bar y_J := \prod_{i \in J} \bar x_i$ satisfies $M_t(\bar y)_{J,J'} = \bar y_J \cdot \bar y_{J'}$ (using $x_i^2 = x_i$), so $M_t(\bar y) \succeq 0$. For each constraint $g_\ell(x) = b_\ell - {a^\ell} x$ and a vector $y \in \R^{\P(V)}$, define the shifted vector $g_\ell * y \in \R^{\P(V)}$ by
\begin{align}
    (g_\ell * y)_J := b^\ell y_J - \sum_{i=1}^n a^\ell_i \, y_{J \cup \{i\}}, \qquad J \subseteq V. \label{def:localizing-product}
\end{align}
For $\bar x \in P \cap \{0,1\}^n$, one has $M_t(g_\ell * \bar y) = g_\ell(\bar x)\, M_t(\bar y) \succeq 0$.

Fix an integer $t$ with $1 \leq t \leq n$. The level-$t$ Lasserre relaxation is defined by
\begin{align}
    \widetilde{La}_t(P) := \{y \in \R^{\P_{2(t+1)}(V)} : y_\emptyset = 1, \ M_{t+1}(y) \succeq 0, \ M_t(g_\ell * y) \succeq 0 \ \text{for } \ell = 1,\ldots,m\}, \label{def:lasserre-lifted}
\end{align}
and its projection onto $\R^n$ is
\begin{align}
    La_t(P) := \{x \in \R^n : x_i = y_{\{i\}}, \ i \in V, \text{ for some } y \in \widetilde{La}_t(P)\}. \label{def:lasserre-projected}
\end{align}
We have $P \supseteq La_1(P) \supseteq La_2(P) \supseteq \cdots \supseteq La_n(P) = \operatorname{conv}(S)$ (\cite{lasserre2001explicit}, \cite{laurent2003comparison}).

Throughout this section, we assume that $g_1,\ldots,g_m$ includes the $2n$ bound constraints $x_i \geq 0$ and $1 - x_i \geq 0$ for $i \in V$. This assumption is without loss of generality because the conditions $M_t(x_i * y) \succeq 0$ and $M_t((1-x_i) * y) \succeq 0$ are implied by $M_{t+1}(y) \succeq 0$ (\cite{laurent2003comparison}, Lemma~5).

\subsubsection{The Lasserre Hierarchy as an \texorpdfstring{$\L$}{L}-Closure} \label{subsec:lasserre-as-L}

As with the Lov\'asz--Schrijver procedure, each positive semidefinite constraint in \eqref{def:lasserre-lifted} can be imposed via intersection cuts. For each $\ell \in \{1,\ldots,m\}$ and each pair of disjoint subsets $J_1, J_2 \subseteq V$ with $|J_1| + |J_2| \leq t$, the product $g_\ell(x) \cdot \prod_{i \in J_1} x_i \cdot \prod_{j \in J_2}(1-x_j) \geq 0$ is valid for $S$. After linearization by replacing $\prod_{i \in J} x_i$ by $y_J$ and using $x_i^2 = x_i$, these products, together with the linearizations of $\prod_{i \in J_1} x_i \cdot \prod_{j \in J_2}(1-x_j) \geq 0$ for $|J_1| + |J_2| \leq t+1$, define a polyhedron $\widehat{La}_t(P) \subseteq \R^{\P_{2(t+1)}(V)}$ with $y_\emptyset = 1$. Define the closed convex sets
\begin{align}
    T^t_{\mathrm{mom}} := \{y \in \R^{\P_{2(t+1)}(V)} : M_{t+1}(y) \succeq 0\}, \qquad T^t_\ell := \{y \in \R^{\P_{2(t+1)}(V)} : M_t(g_\ell * y) \succeq 0\}.
\end{align}

\begin{lemma} \label{lemma:lasserre-psd-strengthen}
    For each closed convex set $T \in \{T^t_{\mathrm{mom}}, T^t_1, \ldots, T^t_m\}$, there exists a family $\L_T$ of closed halfspaces in $\R^{\P_{2(t+1)}(V)}$ such that $R_{\L_T}(C) = C \cap T$ for every convex set $C \subseteq \R^{\P_{2(t+1)}(V)}$.
\end{lemma}

\begin{proof}
    The argument is identical to that used in the proof of Lemma~\ref{lemma:ls-psd-strengthen}. $T$ is closed and convex, and taking $\L_T$ to be all closed halfspaces whose interior is disjoint from $T$ gives $R_{\L_T}(C) = C \cap T$.
\end{proof}

Let $\L^*_{\mathrm{Las}} := \L_{T^t_{\mathrm{mom}}} \cup \L_{T^t_1} \cup \cdots \cup \L_{T^t_m}$. Then
\begin{align}
    R_{\L^*_{\mathrm{Las}}}(\widehat{La}_t(P)) = \widehat{La}_t(P) \cap T^t_{\mathrm{mom}} \cap T^t_1 \cap \cdots \cap T^t_m = \widetilde{La}_t(P). \label{eq:lasserre-as-L}
\end{align}

\subsubsection{The Hereditary Property} \label{subsec:lasserre-hered}

Consider a face $F$ of $P$ exposed by the valid inequality $cx \leq \delta$. The Lasserre relaxation $\widetilde{La}_t(F)$ is defined using the constraints $g_1,\ldots,g_m$ together with the additional equality $\delta - cx = 0$, and therefore requires $M_t((\delta - cx) * y) = 0$ (since both $M_t((\delta - cx)*y) \succeq 0$ and $M_t((cx - \delta)*y) \succeq 0$ must hold). This constraint is not present in $\widetilde{La}_t(P)$. Nevertheless, the following lemma shows that it is automatically satisfied on the face. Since $\delta - cx \geq 0$ is valid for $P$, the semidefinite constraints in $\widetilde{La}_t(P)$ already imply $M_t((\delta - cx)*y) \succeq 0$, and the binding equalities on the face force its diagonal to zero, which forces the entire matrix to vanish.

\begin{lemma} \label{lemma:lasserre-face-redundancy}
    Let $g(x) = \delta - cx$ be a linear function with $g(x) \geq 0$ valid for $P$. If $y \in \widetilde{La}_t(P)$ satisfies $(g * y)_J = 0$ for all $J \in \P_t(V)$, then $M_t(g * y) = 0$.
\end{lemma}

\begin{proof}
    Since $g(x) \geq 0$ is valid for $P = \{x \in \R^n : g_\ell(x) \geq 0, \ \ell = 1,\ldots,m\}$, Farkas' lemma gives nonnegative scalars $u_0, u_1, \ldots, u_m$ such that
    \begin{align}
        g(x) = u_0 + \sum_{\ell=1}^m u_\ell\, g_\ell(x). \label{eq:farkas-gmplus}
    \end{align}
    The vector $g * y$ depends linearly on the coefficients of $g$ by \eqref{def:localizing-product}, and the constant function $1$ satisfies $1 * y = y$. Hence \eqref{eq:farkas-gmplus} yields
    \begin{align}
        M_t(g * y) = u_0\, M_t(y) + \sum_{\ell=1}^m u_\ell\, M_t(g_\ell * y). \label{eq:Mt-decomp}
    \end{align}
    From $y \in \widetilde{La}_t(P)$ we have $M_t(g_\ell * y) \succeq 0$ for $\ell = 1,\ldots,m$, and $M_t(y) \succeq 0$ as a principal submatrix of $M_{t+1}(y) \succeq 0$. Since all coefficients in \eqref{eq:Mt-decomp} are nonnegative, $M_t(g * y) \succeq 0$. The diagonal entries of $M_t(g * y)$ are $(g*y)_J$ for $J \in \P_t(V)$, which are $0$ by the hypothesis, and a positive semidefinite matrix with zero diagonal is the zero matrix. Hence $M_t(g * y) = 0$.
\end{proof}

\begin{corollary}[Lasserre hierarchy] \label{cor:lasserre-hered}
    Let $P \subseteq [0,1]^n$ be a polyhedron defined by linear inequalities $g_\ell(x) = b_\ell - {a^\ell} x \geq 0$ for $\ell = 1,\ldots,m$, and let $F$ be a face of $P$. Fix $1 \leq t \leq n$. Then
    \begin{align}
        La_t(F) = La_t(P) \cap F. \notag
    \end{align}
\end{corollary}

\begin{proof}
    Write $F = P \cap \{x : cx = \delta\}$, where $cx \leq \delta$ is valid for $P$, and set $g_{m+1}(x) := \delta - cx$. Then
    \begin{align}
        \widetilde{La}_t(F) = \widetilde{La}_t(P) \cap \{y : M_t(g_{m+1} * y) = 0\}, \label{eq:La-F-lifted}
    \end{align}
    since imposing both $M_t(g_{m+1}*y) \succeq 0$ and $M_t((-g_{m+1})*y) \succeq 0$ is equivalent to $M_t(g_{m+1}*y) = 0$.

    As in Corollary~\ref{cor:sa-hered}, the linearized products of $g_{m+1}(x) \geq 0$ and $-g_{m+1}(x) \geq 0$ with bound-factor products of order at most $t$ give
    \begin{align}
        \widehat{La}_t(F) = \widehat{La}_t(P) \cap \{y : \mathcal{G}_{J_1, J_2}(y) = 0 \text{ for all pairs } (J_1, J_2) \text{ of order} \leq t\}, \label{eq:hat-La-F}
    \end{align}
    where $\mathcal{G}_{J_1, J_2}(y)$ denotes the linearization of $(\delta - cx)\prod_{i \in J_1}x_i \prod_{j \in J_2}(1-x_j)$.  This exhibits $\widehat{La}_t(F)$ as a face of $\widehat{La}_t(P)$. 

    Applying Theorem~\ref{thm:L-closure-hered} with the family $\L^*_{\mathrm{Las}}$ and the face $\widehat{La}_t(F)$ of $\widehat{La}_t(P)$ gives
    \begin{align}
        R_{\L^*_{\mathrm{Las}}}(\widehat{La}_t(F)) = R_{\L^*_{\mathrm{Las}}}(\widehat{La}_t(P)) \cap \widehat{La}_t(F) = \widetilde{La}_t(P) \cap \widehat{La}_t(F). \label{eq:lasserre-lifted-hered}
    \end{align}
    We claim that $\widetilde{La}_t(P) \cap \widehat{La}_t(F) = \widetilde{La}_t(F)$. The inclusion $\widetilde{La}_t(F) \subseteq \widetilde{La}_t(P) \cap \widehat{La}_t(F)$ is immediate. For the reverse, let $y \in \widetilde{La}_t(P) \cap \widehat{La}_t(F)$. The binding equalities $\mathcal{G}_{J_1, \emptyset}(y) = 0$ for $|J_1| \leq t$ give $(g_{m+1}*y)_J = 0$ for all $J \in \P_t(V)$. Lemma~\ref{lemma:lasserre-face-redundancy} then gives $M_t(g_{m+1}*y) = 0$, so $y \in \widetilde{La}_t(F)$ by \eqref{eq:La-F-lifted}.

    Projection to the $x$-space mirrors Corollary~\ref{cor:sa-hered}. For $La_t(F) \subseteq La_t(P) \cap F$, any $\bar y \in \widetilde{La}_t(F)$ lies in $\widetilde{La}_t(P)$ (so $\bar x \in La_t(P)$) and in $\widehat{La}_t(F)$, which by \eqref{eq:hat-La-F} gives $\delta - c\bar x = 0$, so $\bar x \in F$. For the reverse, let $\bar x \in La_t(P) \cap F$ with $\bar y \in \widetilde{La}_t(P)$ projecting to $\bar x$. The polynomial-identity argument of Corollary~\ref{cor:sa-hered} shows $\mathcal{G}_{J_1, J_2}(\bar y) = 0$ for all pairs of order $\leq t$, so $\bar y \in \widehat{La}_t(F)$ by \eqref{eq:hat-La-F}. Then \eqref{eq:lasserre-lifted-hered} and the claim give $\bar y \in \widetilde{La}_t(F)$, hence $\bar x \in La_t(F)$.
\end{proof}

\section{Cutting-Plane Procedures Without the Hereditary Property} \label{sec:non-hered}

Theorem~\ref{thm:L-closure-hered} shows that any cutting-plane procedure realized as an $\L$-closure for a family $\L$ common to the underlying set and its faces satisfies the hereditary property automatically. In this section we present four classical procedures which cannot be realized as such $\L$-closures, and for which the hereditary property fails, namely the closure of Gomory's fractional cuts, derived from either all bases or feasible bases only, the mixed-integer Chv\'atal closure, the $+$-cut closure of \cite{pokutta2010rank}, and the closure of Dantzig cuts derived from feasible bases. By contrast, we will show in Section~\ref{sec:corner-polyhedron} that the closure of the Dantzig cuts derived from all bases (feasible and infeasible) has the hereditary property.

The Gomory fractional and Dantzig cut procedures are defined for pure integer programs in standard form. For these, let $P := \{x \in \R_+^n : Ax = b\}$ be a rational polyhedron with $A \in \Q^{m \times n}$ of rank $m$ and $b \in \Q^m$, and let $S := P \cap \Z^n$. Denoting by $\bar a_{ij}$ the coefficients of the simplex tableau associated with a (possibly infeasible) basis $B$ of $A$, we may rewrite the system $Ax = b$ as
\begin{align}
    x_i = \bar{b}_i - \sum_{j \in N} \bar{a}_{ij} x_j \ \text{ for } i \in B, \label{eq:tableau}
\end{align}
where $N = \{1,\ldots,n\} \setminus B$ denotes the indices of the non-basic variables. The corresponding basic solution $x^B$ is given by $x^B_i = \bar b_i$ for $i \in B$, $x^B_j = 0$ for $j \in N$. The counterexamples for these two procedures involve a facet of $P$ of the form $F = P \cap \{x : x_t = 0\}$. Since $x_t \geq 0$ is facet defining, the row vector ${e^t}^\top$ is linearly independent of the rows of $A$, and $F$ is again a polyhedron in standard form,
\begin{align}
    F = \bigl\{x \in \R_+^n : A_F x = b_F\bigr\}, \quad \text{where } A_F := \begin{pmatrix} A \\ {e^t}^\top \end{pmatrix} \text{ and } b_F := \begin{pmatrix} b \\ 0 \end{pmatrix}, \notag
\end{align}
to which the procedures apply through this description. Both procedures generate their cuts from the simplex tableaux of the defining system, and are therefore intrinsically tied to its description. The cuts of $F$ are generated by the bases of $A_F$ rather than by those of $A$, so the closures of $P$ and of its faces are not realized as $\L$-closures for a family of lattice-free sets common to both, and Theorem~\ref{thm:L-closure-hered} does not apply.

\subsection{A Counterexample for Gomory Fractional Cuts} \label{subsec:counter-gfc}

The \emph{Gomory fractional cut} (\cite{gomory1958outline}) is a classical cutting plane for pure integer programs in standard form. The cutting-plane algorithm based on these cuts, Gomory's lexicographic method, solves pure integer programs in a finite number of iterations (\cite{gomory1963algorithm}). For a basis $B$ of $A$ and a row $i \in B$ with $\bar b_i \notin \Z$, the Gomory fractional cut from row $i$ is
\begin{align}
    \sum_{j \in N} \frac{f_j}{f_0}\, x_j \geq 1, \label{eq:gfc-cut}
\end{align}
where $f_0 := \bar b_i - \floor{\bar b_i} > 0$ and $f_j := \bar a_{ij} - \floor{\bar a_{ij}}$ for $j \in N$. It is a valid inequality for $S$. Indeed, for $x \in S$, row $i$ of the tableau \eqref{eq:tableau} gives $\sum_{j \in N} f_j x_j = (\bar b_i - x_i) - \sum_{j \in N} \floor{\bar a_{ij}} x_j = f_0 + k$ for some $k \in\Z$. Since $f_0 \in (0,1)$ and $\sum_{j \in N} f_j x_j \geq 0$, this implies $\sum_{j \in N} f_j x_j \geq f_0$.

The \emph{Gomory fractional closure} of $P$ is the intersection of $P$ with the Gomory fractional cuts generated by all bases,
\begin{align}
    G_{B}(P) := P \cap \bigcap_{\substack{B \text{ basis of } A \\ x^{B} \notin \Z^n}} \bigl\{x \in \R^n : x \text{ satisfies \eqref{eq:gfc-cut} for all } i \in B \text{ with } \bar b_i \notin \Z\bigr\}. \label{eq:gfc-closure}
\end{align}
As in Gomory's algorithm, one may instead generate the cuts from feasible bases only. We denote by $G_{FB}(P)$ the closure defined as in \eqref{eq:gfc-closure}, with the intersection restricted to the bases additionally satisfying $x^{B} \geq 0$. The next example exhibits a polytope $P$ and a facet $F$ of $P$ for which both closures fail the hereditary property.

\begin{example} \label{ex:gfc}
    Consider the polytope
    \begin{align}
        P := \{x \in \R_+^2 : 2x_1 + x_2 \leq 4, \ 3x_1 + 3x_2 \geq 1\},
    \end{align}
    whose standard form is $P = \{x \in \R_+^4 : 2x_1 + x_2 + x_3 = 4, \ 3x_1 + 3x_2 - x_4 = 1\}$, and its facet
    \begin{align}
        F = P \cap \{x : x_2 = 0\}.
    \end{align}
    Then $G_B(F) \neq G_B(P) \cap F$ and $G_{FB}(F) \neq G_{FB}(P) \cap F$.
\end{example}

The fractional feasible basic solutions of $P$ are $(0, \frac{1}{3}, \frac{11}{3}, 0)$ and $(\frac{1}{3}, 0, \frac{10}{3}, 0)$, and the fractional infeasible basic solution is $(\frac{11}{3}, -\frac{10}{3}, 0, 0)$. The basis $\{2, 3\}$, at the vertex $(0, \frac{1}{3}, \frac{11}{3}, 0)$, has the tableau rows
\begin{align}
    x_2 = \tfrac{1}{3} - x_1 + \tfrac{1}{3} x_4, \qquad x_3 = \tfrac{11}{3} - x_1 - \tfrac{1}{3} x_4, \notag
\end{align}
whose Gomory fractional cuts are $2 x_4 \geq 1$ and $\frac{1}{2} x_4 \geq 1$, respectively. The remaining fractional bases do not produce any different Gomory fractional cuts. Both rows of the basis $\{1,3\}$ yield $2x_4 \geq 1$, and both rows of the infeasible basis $\{1,2\}$ yield $\frac{1}{2} x_4 \geq 1$. Observe that the cut $\frac{1}{2} x_4 \geq 1$ dominates $2 x_4 \geq 1$. In the original variables, the cut $\frac{1}{2}x_4 \geq 1$ reads $x_1 + x_2 \geq 1$. 
On $F$, the cuts $\frac{1}{2}x_4 \geq 1$ and $2x_4 \geq 1$ read $x_1 \geq 1$ and $x_1 \geq \frac{1}{2}$ in the original variables, so
\begin{align}
    G_{B}(P) \cap F = G_{FB}(P) \cap F = \{x \in \R^2_+: 1 \leq x_1 \leq 2, \ x_2 = 0\}. \notag
\end{align}

The system defining $F$ has three bases, each containing $x_2$. The bases $\{1,2,4\}$ and $\{2,3,4\}$ have the integral basic solutions $(2,0,0,5)$ and $(0,0,4,-1)$ and generate no cuts. The only remaining basis is $\{1,2,3\}$, with the feasible basic solution $(\frac{1}{3}, 0, \frac{10}{3}, 0)$ and tableau rows $x_1 = \frac{1}{3} + \frac{1}{3}x_4$ and $x_3 = \frac{10}{3} - \frac{2}{3}x_4$, both of which yield the single Gomory fractional cut $2x_4 \geq 1$. In the original variables, this cut is $x_1 \geq \frac{1}{2}$ on $F$. Hence, under either convention,
\begin{align}
    G_{B}(F) = G_{FB}(F) = \{x \in \R^2_+: \tfrac{1}{2} \leq x_1 \leq 2, \ x_2 = 0\} \neq G_B(P) \cap F. \notag
\end{align}

\begin{figure}[ht!]
    \centering
    \includegraphics[width=0.5\linewidth]{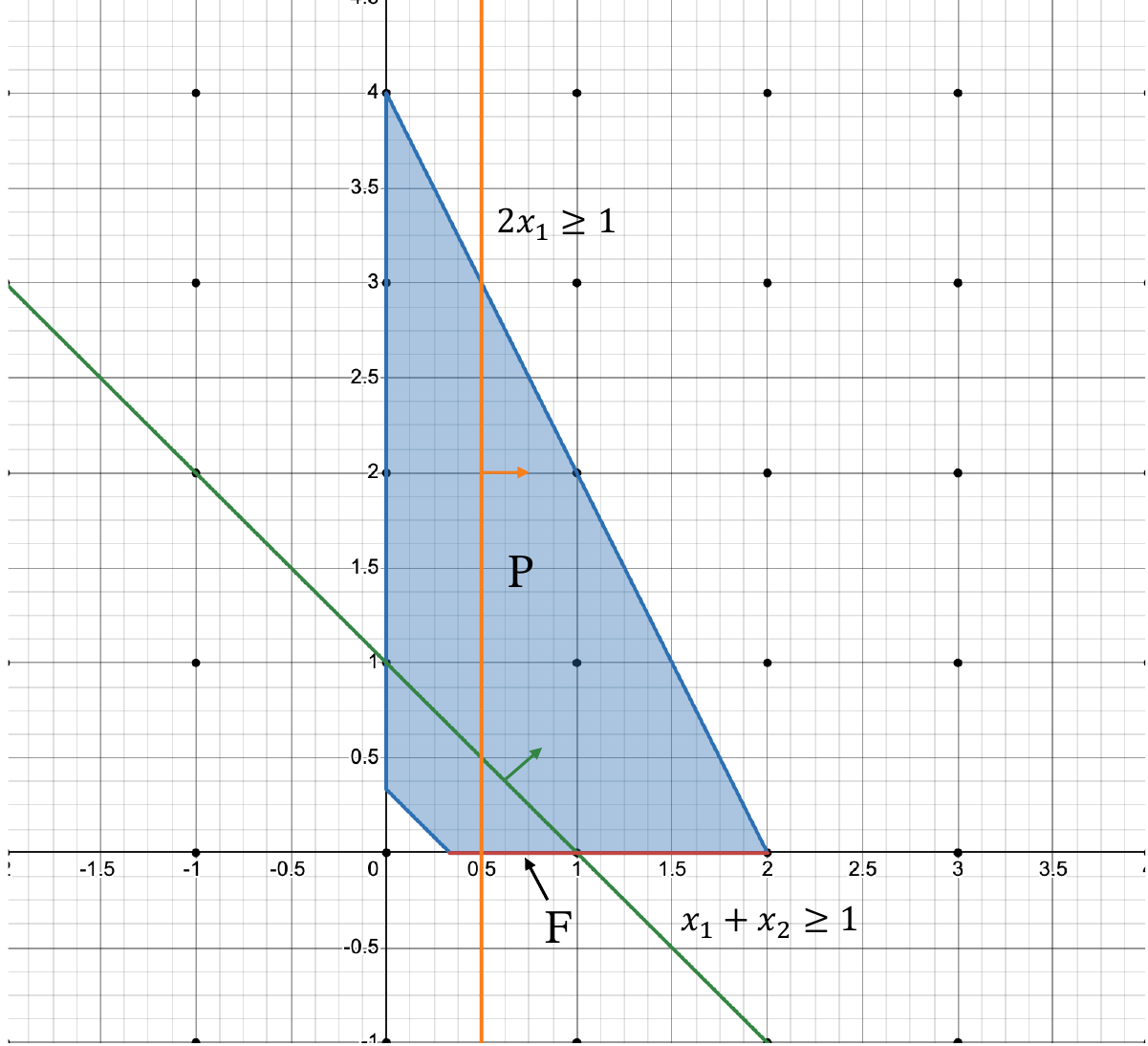}
    \caption{A 2-dimensional example where the Gomory fractional cuts do not satisfy the hereditary property.}
    \label{fig:gomory_example}
\end{figure}

The failure is driven by the bases $\{1,2\}$ and $\{2,3\}$ of $P$, which produce the strongest Gomory fractional cut, $x_1 \geq 1$, on $F$. However, the corresponding solutions $(\frac{11}{3}, -\frac{10}{3}, 0, 0)$ and $(0, \frac{1}{3}, \frac{11}{3}, 0)$ violate $x_2 = 0$, so they are not basic solutions of the system defining $F$ and cannot be used to derive this cut (see Figure \ref{fig:gomory_example}).

\begin{remark} \label{rmk:gfc-chvatal}
    Every Gomory fractional cut is a Chv\'atal inequality of $P$ (see \cite{conforti2014integer}, Section~5.2), so the Gomory fractional cut closure contains the Chv\'atal closure. The containment is strict in general (\cite{cornuejols2001elementary}). While the Chv\'atal closure satisfies the hereditary property for faces of rational polyhedra in the pure integer setting (\cite{schrijver1980cutting}), Example\ref{ex:gfc} shows that its relaxations $G_B(P)$ and $G_{FB}(P)$ do not.
\end{remark}

\subsection{A Counterexample for the Mixed-Integer Chv\'atal Closure} \label{subsec:counter-mixed-chvatal}

We consider a mixed-integer extension of the classical Chv\'atal closure (\cite{chvatal1973edmonds}) and exhibit a polytope $P$ and a face $F$ for which the hereditary property fails to hold.

Recall from Section~\ref{subsec:notation} the integer set $I$ and the continuous set $C := \{1,\ldots,n\} \setminus I$. Let $Q = \{x \in \R^n : Ax \leq b\}$ be a rational polyhedron with mixed-integer feasible region $S = \{x \in Q : x_j \in \Z, \ j \in I\}$. A \emph{mixed-integer Chv\'atal inequality} for $Q$ is an inequality $\pi x \leq \pi_0$ with $\pi \in \Z^n$ such that there exists $u \geq 0$ with $uA = \pi$, $uA_C = 0$, $uA_I \in \Z^{|I|}$, and $\pi_0 \geq \floor{ub}$. The \emph{mixed-integer Chv\'atal closure} $Q^{Ch}$ of $Q$ is the intersection of $Q$ with all such inequalities. 

Note that the structural mechanism of Section~\ref{sec:intersection-cuts} does not apply. Since mixed-integer Chv\'atal inequalities are derived from the inequalities valid for the underlying set, applying the procedure to a face $F$ generates cuts that cannot necessarily be obtained from any family of lattice-free convex sets used to obtain the closure of $Q$, so the procedure is not an $\L$-closure for a family common to $Q$ and its faces.

\begin{example}
    Let $I = \{1\}$, $C = \{2\}$, and $S = P \cap (\Z \times \R)$ for the polytope
    \begin{align}
        P := \{x \in \R^2: x_1 + x_2 \leq 1, \ -x_1 + x_2 \leq 0, \ x_2 \geq 0\}.
    \end{align}
    Consider its face
    \begin{align}
        F = P \cap \{x : x_2 = \tfrac{1}{2}\} = \{(\tfrac{1}{2},\tfrac{1}{2})\}.
    \end{align}
    Then $F^{Ch} \neq P^{Ch} \cap F$.
\end{example}

It is straightforward to check that $F^{Ch} = \emptyset$, since $F$ is a single fractional point in the integer-restricted variable $x_1$, while $P^{Ch} = P$, since no nontrivial mixed-integer Chv\'atal inequality cuts off any point of $P$. Therefore $\emptyset = F^{Ch} \neq P^{Ch} \cap F = F$.

\subsection{A Counterexample for \texorpdfstring{$+$}{+}-Cuts} \label{subsec:counter-pluscut}

The $+$-cut procedure was introduced by \cite{pokutta2010rank} as a cutting-plane technique in the 0-1 setting. Unlike convexification-based methods, it is combinatorial in nature and relies on a structured family of disjunctions over logarithmically sized subsets of variables.

Let $P \subseteq [0,1]^n$ be a polytope and $S := P \cap \{0,1\}^n$. Let $\tilde J \subseteq \{1,\ldots,n\}$ with $|\tilde J| \leq \ceil{\log n}$, and let $I \subseteq \tilde I \subseteq \{1,\ldots,n\}$ with $\tilde I \cap \tilde J = \emptyset$.
If there exists $\epsilon > 0$ such that, for all $J \subseteq \tilde J$, the inequality
\begin{align}
    \sum_{i \in I} x_i + \sum_{i \in \tilde I \setminus I} (1 - x_i) + \sum_{i \in J} x_i + \sum_{i \in \tilde J \setminus J} (1 - x_i) \geq \epsilon
\end{align}
is valid for $P$, then the inequality
\begin{align}
    \sum_{i \in I} x_i + \sum_{i \in \tilde I \setminus I} (1 - x_i) \geq 1
\end{align}
is called a \emph{$+$-cut} for $P$.
The \emph{$+$-cut closure} of $P$, denoted $P^+$, is defined as the intersection of all $+$-cuts valid for $P$.

As with the mixed-integer Chv\'atal closure and the Gomory fractional cut closure, the mechanism of Section~\ref{sec:intersection-cuts} does not apply to $+$-cuts. A $+$-cut is derived from inequalities valid for $P$, so applying the procedure to a face $F$ generates new cuts that cannot be derived from any family of lattice-free convex sets that may produce the closure of $P$. The procedure is therefore not an $\L$-closure for a family common to $P$ and its faces.

We now exhibit a polytope $P$ and a face $F$ for which $F^+ \neq P^+ \cap F$.

\begin{example}
    Consider the polytope
    \begin{align}
        P := \{x \in [0,1]^3: x_1 + x_2 \geq 1, \ x_1 + x_3 \geq 1, \ x_2 + x_3 \geq 1\}
    \end{align}
    and its face
    \begin{align}
        F = P \cap \{x : x_1 + x_2 = 1\}.
    \end{align}
    Then $F^{+} \neq P^{+} \cap F$.
\end{example}

Observe that
\begin{align}
    P^+ = P = \operatorname{conv}\big( (0,1,1),\ (1,1,1),\ (1,1,0),\ (1,0,1),\ (\tfrac{1}{2},\tfrac{1}{2},\tfrac{1}{2}) \big),
\end{align}
because the following inequalities are all of the $+$-cuts of $P$:
\begin{align}
    x_1 + x_2 \geq 1 \notag \\
    x_1 + x_3 \geq 1 \notag \\
    x_2 + x_3 \geq 1 \notag \\
    x_1 + x_2 + x_3 \geq 1 \notag \\
    x_1 + (1-x_2) + x_3 \geq 1 \notag \\
    (1-x_1) + x_2 + x_3 \geq 1 \notag \\
    x_1 + x_2 + (1-x_3) \geq 1 \notag
\end{align}
However, $F^+ = F \cap \{x: x_3 = 1\} = \operatorname{conv}((0,1,1),(1,0,1))$ as the inequality $x_3 \geq 1$ is a $+$-cut of $F$. Therefore, $F = P^+ \cap F \neq F^+$.

\subsection{A Counterexample for Dantzig Cuts from Feasible Bases} \label{subsec:counter-dantzig}

We conclude with the cutting-plane procedure introduced by \cite{dantzig1959note}, one of the earliest methods developed for pure integer programs. For a basis $B$ of $A$ with $x^B \notin \Z^n$, the associated \emph{Dantzig cut} is
\begin{align}
    \sum_{j \in N} x_j \geq 1. \label{def:dantzig-cut}
\end{align}
This inequality removes the fractional basic solution without violating any point of $S$ because a point $x \in S$ with $x_j = 0$ for all $j \in N$ would equal $x^B \notin \Z^n$.

Dantzig originally defined this procedure for the pure integer setting, but it can be extended naturally to mixed-integer programs. In particular, let $I$ and $C$ denote the index sets of integer and continuous variables, respectively. In this case, one may omit the tableau equations corresponding to basic variables in $C$, since each such variable $x_i$ appears in only one equality and does not affect integrality constraints. Hence, without loss of generality, all basic variables may be assumed to be integer variables when generating Dantzig cuts.

Although Dantzig provides a procedural method for deriving such cuts, the notion of a Dantzig cut closure is not well defined in the literature. Since Dantzig's algorithm generates cuts at basic feasible solutions, a natural first definition is the closure derived from the cuts \eqref{def:dantzig-cut} associated with feasible bases,
\begin{align}
    D_{FB}(P) := P \cap \Big\{x \in \R^n: \sum_{j \in N} x_j \geq 1 \text{ for all bases } B \text{ of } A \text{ such that } x^{B} \geq 0 \text{ and } x^{B} \notin \Z^n \Big\}. \label{eq:dantzig-feasible-closure}
\end{align}
In contrast with the Gomory fractional cut, a cutting-plane algorithm based on Dantzig cuts, which are generated at basic feasible solutions, cannot, in general, solve a pure integer program in a finite number of iterations (\cite{gomory1963convergence}). The next example shows that the closure $D_{FB}$ also fails to satisfy the hereditary property.

\begin{example} \label{ex:dantzig-feasible}
    Consider the polytope
    \begin{align}
        P := \{x \in \R_+^2 : 2x_1 + 2x_2 \leq 1, \ 5x_2 \leq 2\},
    \end{align}
    whose standard form is $P = \{x \in \R_+^4 : 2x_1 + 2x_2 + x_3 = 1, \ 5x_2 + x_4 = 2\}$, and its facet
    \begin{align}
        F = P \cap \{x : x_1 = 0\}.
    \end{align}
    Then $D_{FB}(F) \neq D_{FB}(P) \cap F$.
\end{example}

In the standard form, $S = P \cap \Z^4 = \{(0,0,1,2)\}$. The polytope $P$ has one integral basic feasible solution $(0,0,1,2)$, and three fractional basic feasible solutions $(\frac{1}{10}, \frac{2}{5}, 0, 0)$, $(\frac{1}{2}, 0, 0, 2)$, and $(0, \frac{2}{5}, \frac{1}{5}, 0)$. The fractional basic feasible solutions yield the Dantzig cuts $x_3 + x_4 \geq 1$, $x_2 + x_3 \geq 1$, and $x_1 + x_4 \geq 1$, respectively. On $F$ and in the original variables, these cuts read $x_2 \leq \frac{2}{7}$, $x_2 \leq 0$, and $x_2 \leq \frac{1}{5}$, so $D_{FB}(P) \cap F = \{(0,0,1,2)\}$. The system defining $F$ has a single fractional basic feasible solution, namely $(0, \frac{2}{5}, \frac{1}{5}, 0)$ with basis $\{1,2,3\}$, whose Dantzig cut is $x_4 \geq 1$ and reads $x_2 \leq \frac{1}{5}$ in the original variables. Hence
\begin{align}
    D_{FB}(F) = \bigl\{x\in \R^2_+: x_1 = 0, \ x_2 \leq \tfrac{1}{5}\bigr\} \neq \{(0,0)\} = D_{FB}(P) \cap F. \notag
\end{align}
The cut $x_2 + x_3 \geq 1$ responsible for the failure is generated by the vertex $(\tfrac{1}{2}, 0, 0, 2)$, which does not lie on $F$. The failure, however, appears to be tied to the restriction to feasible bases rather than to the Dantzig cut itself. The basis $\{1,2,4\}$ of $A_F$, whose basic solution $(0, \tfrac{1}{2}, 0, -\tfrac{1}{2})$ is infeasible, generates the additional cut $x_3 \geq 1$, which gives $x_2 \leq 0$ on $F$ in the original variables, and restores the strongest Dantzig cut on $F$ in this example. In Section~\ref{sec:corner-polyhedron} we show that this is not specific to this instance, by proving that the Dantzig closure derived from all bases, feasible or not, satisfies the hereditary property (Corollary~\ref{cor:dantzig-hered}).

\section{The Hereditary Property via the Corner Polyhedron} \label{sec:corner-polyhedron}

The counterexamples of Section~\ref{sec:non-hered} show that the hereditary property can fail for several well-known procedures. In the case of the Dantzig closure, the failure can hinge on which bases are permitted to generate cuts. In this section we develop a framework that accounts for this behavior.

The Gomory fractional and Dantzig cuts are valid inequalities for Gomory's \emph{corner polyhedron} (\cite{gomory1969some}), a relaxation of the integer hull of $P$ associated with a basis of the constraint matrix. Cuts of this kind are indexed by bases, and a face of a polyhedron in standard form has bases of its own. Theorem~\ref{thm:corner-cut-hered} below gives a sufficient condition, relating the cuts that a procedure generates for a polyhedron to those it generates for a face, under which the resulting closure satisfies the hereditary property. The condition has two halves. Each procedure of Section~\ref{sec:non-hered} that generates cuts of this kind satisfies one half and fails the other, which locates the failure of the hereditary property precisely. We show that the Dantzig closure derived from all bases satisfies both, and the hereditary property follows.

We continue in the standard-form setting. Let $P := \{x \in \R_+^n: Ax = b\}$ be a rational polyhedron with $A \in \Q^{m \times n}$ of rank $m$ and $b \in \Q^m$. For a (possibly infeasible) basis $B$ of $A$ with nonbasic index set $N$, we follow the simplex tableau notation of \eqref{eq:tableau}.

Gomory defines the set
$$P(B) := \{x \in \R^n: x_i = \bar{b}_i - \sum_{j \in N} \bar{a}_{ij} x_j \text{ for } i \in B, \ x_j \geq 0 \text{ for } j \in N\}.$$
\noindent The corner polyhedron relative to the basis $B$, $\operatorname{corner}(B)$, is the convex hull of the integer points in $P(B)$,
$$\operatorname{corner}(B) := \operatorname{conv}(P(B) \cap \Z^n).$$

The following lemma characterizes all nontrivial valid inequalities for $\operatorname{corner}(B)$ (see, e.g., \cite{conforti2014integer}, Lemma 6.4).
\begin{lemma} \label{lemma:corner-valid}
    Assume $\operatorname{corner}(B)$ is nonempty. Every nontrivial valid inequality for $\operatorname{corner}(B)$ can be written in the form $\sum_{j \in N} \gamma_j x_j \geq 1$ where $\gamma_j \geq 0$ for all $j \in N$.
\end{lemma}

Furthermore, by \cite{conforti2014integer}, Theorem~6.12, the nontrivial valid inequalities for $\operatorname{corner}(B)$ are intersection cuts. Hence any closure derived from them is again an instance of an $\L$-closure, but with a family $\L$ indexed by the bases of the underlying polyhedron.

\subsection{Faces and Basis Extension} \label{subsec:basis-extension}

Let $J \subseteq \{1, \ldots, n\}$ index a minimal set of facet-defining inequalities such that $P = \{x \in \R^n: Ax = b, \ x_j \geq 0, \ j \in J \}$. Let $F$ be a proper face of $P$ characterized by $F := \{x \in \R^n: A_Fx = b_F, x \geq 0\}$.
\begin{lemma} \label{lemma:A_F-rank}
    We may assume that $A_F$ has rank $m+d$ where $d = \operatorname{dim}(P) - \operatorname{dim}(F)$.
\end{lemma}
\begin{proof}
    If $F$ is a proper face of $P$, there exists a facet $F_1$ of $P$ that contains $F$. Let $F_1 = P \cap \{x \in \R^n : x_t=0\}$ for some $t \in J$. Let $x_j \geq 0$ for $j \in J_1$ be a minimal set of facet-defining inequalities for the polyhedron $F_1$. Then $F_1 = \{x \in \R^n: A_{F_1}x=b_{F_1}, \ x_j \geq 0 \ \forall j \in J_1\}$, where $A_{F_1} = \begin{pmatrix}A \\ {e^t}^\top\end{pmatrix}$ and $b_{F_1} = \begin{pmatrix}b \\ 0\end{pmatrix}$. Since $x_t \geq 0$ is a nonredundant facet-defining inequality of $P$, ${e^t}^\top$ is linearly independent of the rows of $A$. Therefore, $\operatorname{rank}(A_{F_1}) = m+1$.

    By induction on $d$, we have that $F = \{x \in \R^n: A_{F}x=b_{F}, \ x_j \geq 0 \ \forall j \in J_d\}$, where $A_F = \begin{pmatrix}A \\ E\end{pmatrix}$, $E$ is a matrix of standard basis row vectors corresponding to variables set to zero in $F$, $\operatorname{rank}(A_F) = m+d$, $b_F = \begin{pmatrix}b \\ \mathbf{0}\end{pmatrix}$,  and $J_d$ indexes a minimal set of facet-defining inequalities in $F$.
\end{proof}

Throughout, we let $D \subseteq J$ denote the set of indices $j$ such that $x_j = 0$ is a constraint in $A_F x = b_F$, so that $|D| = d$ and $E$ consists of the rows ${e^j}^\top$ for $j \in D$.

For a basis $B_F$ of $A_F$, define
$$F(B_F) := \{x \in \R^n: x_i = \tilde{b}_i - \sum_{j \in N_F} \tilde{a}_{ij} x_j \text{ for } i \in B_F, \ x_j \geq 0 \text{ for } j \in N_F\},$$
where $\tilde b$ is the corresponding basic solution in $F$, and $\tilde a_{ij}$ are the coefficients of the simplex tableau of the extended system that defines $F$. Let $\operatorname{corner}(B_F)$ be the convex hull of the integer points in $F(B_F)$.

Consider a basis $B_F$ of $A_F$. The row ${e^j}^\top$ of $E$ has its only nonzero entry in column $j$, so if some $j \in D$ were nonbasic, the submatrix of $A_F$ with columns indexed by $B_F$ would contain a zero row. Therefore, the $d$ variables $x_j$, for $j \in D$, are basic variables in $B_F$. Moreover, $B := B_F \setminus D$ is a basis of $A$ with $|B| = m$. Indeed, after permuting rows and columns, the nonsingular submatrix of $A_F$ with columns indexed by $B_F$ takes the block triangular form $\begin{pmatrix} A_B & A_D \\ 0 & I \end{pmatrix}$, where $A_B$ and $A_D$ denote the submatrices of $A$ with columns indexed by $B$ and by $D$, so $A_B$ is nonsingular. The corresponding sets of non-basic variables satisfy $N = N_F \cup D$. The next lemma shows that, along this correspondence, the simplex tableau entries in the columns indexed by $N_F$ are unchanged.

\begin{lemma} \label{lemma:basis-correspondence}
    Let $B_F$ be a basis of $A_F$ and $B := B_F \setminus D$, with nonbasic index sets $N_F$ and $N = N_F \cup D$, respectively. Then $\tilde b_i = \bar b_i$ and $\tilde a_{ij} = \bar a_{ij}$ for all $i \in B$ and $j \in N_F$. Moreover, $\tilde b_i = 0$ and $\tilde a_{ij} = 0$ for all $i \in D$ and $j \in N_F$. In particular, $x^{B_F} = x^B$.
\end{lemma}

\begin{proof}
    On the affine subspace $\{x \in \R^n : Ax = b\}$, the tableau of $B$ expresses
    \begin{align}
        x_i = \bar b_i - \sum_{j \in N_F} \bar a_{ij} x_j - \sum_{j \in D} \bar a_{ij} x_j \qquad \text{for } i \in B. \notag
    \end{align}
    On the affine subspace $\{x \in \R^n : A_F x = b_F\}$, we additionally have $x_j = 0$ for $j \in D$, so
    \begin{align}
        x_i = \bar b_i - \sum_{j \in N_F} \bar a_{ij} x_j \ \text{ for } i \in B, \qquad x_j = 0 \ \text{ for } j \in D. \notag
    \end{align}
    These equations express the basic variables of $B_F = B \cup D$ in terms of the nonbasic variables indexed by $N_F$. Since the tableau representation with respect to a basis is unique, they constitute the tableau of $B_F$, which proves the stated equalities. Setting the nonbasic variables to zero yields $x^{B_F} = x^B$.
\end{proof}

We present a second structural lemma that goes in the opposite direction, extending a basis of $A$ to a basis of $A_F$ whose basic solution is again fractional.

\begin{lemma} \label{lemma:basis-extension}
    Let $B$ be a basis of $A$ with $x^B \notin \Z^n$, and let $\sum_{j \in N} \gamma_j x_j \geq 1$, with $\gamma_j \geq 0$ for $j \in N$, be a valid inequality for $\operatorname{corner}(B)$ that is not valid for $F$. Then there exists a basis $B_F$ of $A_F$ such that $B \subset B_F$ and $x^{B_F} \notin \Z^n$.
\end{lemma}

\begin{proof}
    Consider the polyhedron $\bar F := \{x \in \R^n: A_Fx=b_F, \ x_j \geq 0 \text{ for } j \in N\}$, which contains $F$. It is pointed, since any $w$ in its lineality space satisfies $Aw = 0$ and $w_j = 0$ for $j \in N$, hence $A_B w_B = 0$ and $w = 0$.

    We first show that $\bar F$ has a vertex $\hat x$ such that $\sum_{j \in N} \gamma_j \hat x_j < 1$.    
     
    Since the cut is not valid for $F$ and $F \subseteq \bar F$, there exists $\bar y \in \bar F$ with $\sum_{j \in N} \gamma_j \bar y_j < 1$. By the decomposition theorem for polyhedra, we may write $\bar y = \sum_{i=1}^{p} \alpha_i v^i + \sum_{t=1}^{q} \beta_t r^t$, where $v^1,\ldots,v^p$ are the vertices of $\bar F$, $r^1,\ldots,r^q$ are its extreme rays, $\sum_{i=1}^{p} \alpha_i = 1$, $\alpha_i \geq 0$, and $\beta_t \geq 0$. Every ray $r^t$ of $\bar F$ satisfies $r^t_j \geq 0$ for $j \in N$, and $\gamma \geq 0$, so $\sum_{j \in N} \gamma_j r^t_j \geq 0$. It follows that $\sum_{j \in N} \gamma_j v^i_j < 1$ for some $i \in \{1,\ldots,p\}$, and we may take $\hat x := v^i$.

    Because $\hat x$ is a vertex of $\bar F$, it satisfies $n$ linearly independent constraints of $\bar F$ with equality. Since the $m+d$ rows of $A_F$ are linearly independent, these constraints may be chosen as the equalities $A_F x = b_F$ together with the equations $x_j = 0$ for $j$ in a set $T \subseteq \{j \in N : \hat x_j = 0\}$ with $|T| = n-m-d$. We show that $B_F := \{1,\ldots,n\} \setminus T$ is the desired basis of $A_F$. The $n \times n$ matrix $\begin{pmatrix} A_F \\ I_T \end{pmatrix}$, where $I_T$ consists of the rows ${e^j}^\top$ for $j \in T$, has the chosen constraints as its rows and is therefore nonsingular. Each row of $I_T$ has a single nonzero entry, so cofactor expansion of the determinant along these rows successively removes them together with the columns indexed by $T$. The columns of $A_F$ indexed by $B_F$ thus form a nonsingular submatrix, and $B_F$ is a basis of $A_F$ with nonbasic index set $T$. Moreover, $x^{B_F} = \hat x$, as both points satisfy the nonsingular system $A_F x = b_F$, $x_j = 0$ for $j \in T$. Since $T \subseteq N$, we have $B \subset B_F$. Note that $B_F$ may be an infeasible basis of $A_F$.

    It remains to verify that $x^{B_F} \notin \Z^n$. If $x^{B_F} \in \Z^n$, then $x^{B_F}_j = \hat x_j \geq 0$ for all $j \in N$, so $x^{B_F} \in P(B) \cap \Z^n \subseteq \operatorname{corner}(B)$, while $\sum_{j \in N} \gamma_j x^{B_F}_j = \sum_{j \in N} \gamma_j \hat x_j < 1$, contradicting the validity of the cut for $\operatorname{corner}(B)$.
\end{proof}

\subsection{Basis-Compatible Families of Corner Cuts} \label{subsec:corner-hered}

We consider cutting-plane procedures of the following general form: for each basis $B$ of $A$ such that $x^B \notin \Z^n$, the procedure generates a (possibly empty) set $\Gamma(B)$ of nonnegative vectors $\gamma = (\gamma_j)_{j \in N}$ that produce valid inequalities for $\operatorname{corner}(B)$, each of the form $\sum_{j \in N} \gamma_j x_j \geq 1$. We call $\Gamma$ a \emph{family of corner cuts}. The procedure defining $\Gamma$ applies to an arbitrary system in standard form. In particular, $\Gamma$ assigns cuts to the bases of $A_F$ for each proper face $F$ of $P$. The \emph{$\Gamma$-closure} of $P$ is
\begin{align}
    P^\Gamma := P \cap \bigcap_{\substack{B \text{ basis of } A \\ x^B \notin \Z^n}} \Bigl\{x \in \R^n : \sum_{j \in N} \gamma_j x_j \geq 1 \ \forall \gamma \in \Gamma(B)\Bigr\}, \label{eq:gamma-closure}
\end{align}
and the $\Gamma$-closure $F^\Gamma$ of a face $F$ is defined analogously, with the bases of $A_F$ in place of the bases of $A$.

We say that a family of corner cuts $\Gamma$ is \emph{basis-compatible} if, for every polyhedron $P$ in standard form and every proper face $F$ of $P$, the following two conditions hold:
\begin{enumerate}[label=(\roman*)]
    \item \emph{(Lifting)} For every basis $B_F$ of $A_F$ with $x^{B_F} \notin \Z^n$ and every $\gamma \in \Gamma(B_F)$, there is a $\gamma' \in \Gamma(B)$, where $B = B_F \setminus D$, such that $\gamma'_j \leq \gamma_j$ for all $j \in N_F$.
    \item \emph{(Restriction)} For every basis $B$ of $A$ with $x^B \notin \Z^n$ and every $\gamma \in \Gamma(B)$ whose cut $\sum_{j \in N} \gamma_j x_j \geq 1$ is not valid for $F$, there exist a basis $B_F$ of $A_F$ with $B \subset B_F$ and $x^{B_F} \notin \Z^n$, and a $\gamma' \in \Gamma(B_F)$ such that $\gamma'_j \leq \gamma_j$ for all $j \in N_F$.
\end{enumerate}
In the lifting condition, $x^B = x^{B_F} \notin \Z^n$ by Lemma~\ref{lemma:basis-correspondence}, so $\Gamma(B)$ is well defined. Informally, the lifting condition states that every cut of the face can be lifted to a cut of the polyhedron that dominates on the face. Likewise, the restriction condition states that every cut of the polyhedron can be restricted to a cut of the face that dominates on the face.

\begin{theorem} \label{thm:corner-cut-hered}
    Let $\Gamma$ be a family of corner cuts and let $F$ be a proper face of $P$.
    \begin{enumerate}[label=(\alph*)]
        \item If $\Gamma$ satisfies the lifting condition, then $P^\Gamma \cap F \subseteq F^\Gamma$.
        \item If $\Gamma$ satisfies the restriction condition, then $F^\Gamma \subseteq P^\Gamma \cap F$.
    \end{enumerate}
    In particular, if $\Gamma$ is basis-compatible, then $F^\Gamma = P^\Gamma \cap F$.
\end{theorem}

\begin{proof}
    (a) Let $y \in P^\Gamma \cap F$, and consider the cut $\sum_{j \in N_F} \gamma_j x_j \geq 1$ generated by some $\gamma \in \Gamma(B_F)$ for a basis $B_F$ of $A_F$ with $x^{B_F} \notin \Z^n$. By the lifting condition, there is a $\gamma' \in \Gamma(B)$, where $B = B_F \setminus D$, with $\gamma'_j \leq \gamma_j$ for all $j \in N_F$, and $y \in P^\Gamma$ satisfies the cut $\sum_{j \in N} \gamma'_j x_j \geq 1$. Since $y \in F$, we have $y_j = 0$ for $j \in D$ and $y \geq 0$, so, using $N = N_F \cup D$,
    \begin{align}
        1 \leq \sum_{j \in N} \gamma'_j y_j = \sum_{j \in N_F} \gamma'_j y_j \leq \sum_{j \in N_F} \gamma_j y_j. \notag
    \end{align}
    Hence $y$ satisfies every cut of $F$, and $y \in F^\Gamma$.

    (b) Since $F^\Gamma \subseteq F \subseteq P$, it suffices to show that every $y \in F^\Gamma$ satisfies the cut $\sum_{j \in N} \gamma_j x_j \geq 1$ for every basis $B$ of $A$ with $x^B \notin \Z^n$ and every $\gamma \in \Gamma(B)$. If the cut is valid for $F$, this is clear. Otherwise, the restriction condition provides a basis $B_F$ of $A_F$ with $B \subset B_F$ and $x^{B_F} \notin \Z^n$, and a $\gamma' \in \Gamma(B_F)$ with $\gamma'_j \leq \gamma_j$ for all $j \in N_F$, and $y$ satisfies the cut $\sum_{j \in N_F} \gamma'_j x_j \geq 1$. Since $N_F \subset N$ and $y \geq 0$,
    \begin{align}
        \sum_{j \in N} \gamma_j y_j \geq \sum_{j \in N_F} \gamma_j y_j \geq \sum_{j \in N_F} \gamma'_j y_j \geq 1. \notag
    \end{align}
\end{proof}

We next classify the procedures of Section~\ref{sec:non-hered} with respect to this framework.

The mixed-integer Chv\'atal and $+$-cut closures fall outside it. The mixed-integer Chv\'atal procedure admits an analog of the Gomory fractional cut through Gomory's mixed-integer inequalities (\cite{gomory1960mixed}), but such cuts may carry negative coefficients. Likewise, rewriting a $+$-cut in the nonbasic variables of a tableau yields negative coefficients from the complemented terms. Neither procedure therefore generates cuts of the form $\sum_{j \in N} \gamma_j x_j \geq 1$ with $\gamma \geq 0$ that, by Lemma~\ref{lemma:corner-valid}, characterizes the nontrivial valid inequalities for $\operatorname{corner}(B)$.

The Gomory fractional and Dantzig cuts, by contrast, fit the framework. The computation of Section~\ref{subsec:counter-gfc} applies verbatim to any $x \in P(B) \cap \Z^n$, so the Gomory fractional cut \eqref{eq:gfc-cut} is valid for $\operatorname{corner}(B)$. The Dantzig cut \eqref{def:dantzig-cut} is likewise valid for $\operatorname{corner}(B)$, since a point $x \in P(B) \cap \Z^n$ with $x_j = 0$ for all $j \in N$ would equal $x^B$. Accordingly, consider a basis $B$ of $A$ with $x^B \notin \Z^n$. Let $\Gamma_{\mathrm{GB}}(B)$ consist of the coefficient vectors $(f_j/f_0)_{j \in N}$ of the Gomory fractional cuts \eqref{eq:gfc-cut} from the rows $i \in B$ with $\bar b_i \notin \Z$, and let $\Gamma_{\mathrm{DB}}(B) := \{\mathbf{1}\}$, where $\mathbf{1} := (1)_{j \in N}$ is the coefficient vector of the Dantzig cut \eqref{def:dantzig-cut}. The two feasible-basis conventions are obtained by discarding the infeasible bases: $\Gamma_{\mathrm{GFB}}(B) := \Gamma_{\mathrm{GB}}(B)$ and $\Gamma_{\mathrm{DFB}}(B) := \Gamma_{\mathrm{DB}}(B)$ if $x^B \geq 0$, and $\Gamma_{\mathrm{GFB}}(B) := \Gamma_{\mathrm{DFB}}(B) := \emptyset$ otherwise. These are all families of corner cuts, and
\begin{align}
    G_{B}(P) = P^{\Gamma_{\mathrm{GB}}}, \qquad G_{FB}(P) = P^{\Gamma_{\mathrm{GFB}}}, \qquad D_{FB}(P) = P^{\Gamma_{\mathrm{DFB}}}. \notag
\end{align}

\begin{lemma} \label{lemma:lifting-taxonomy}
    The families $\Gamma_{\mathrm{GB}}$, $\Gamma_{\mathrm{GFB}}$, and $\Gamma_{\mathrm{DFB}}$ satisfy the lifting condition.
\end{lemma}

\begin{proof}
    Let $B_F$ be a basis of $A_F$ with $x^{B_F} \notin \Z^n$ and let $B = B_F \setminus D$. By Lemma~\ref{lemma:basis-correspondence}, $\tilde b_i = \bar b_i$ and $\tilde a_{ij} = \bar a_{ij}$ for all $i \in B$ and $j \in N_F$, $\tilde b_i = 0$ for $i \in D$, and $x^{B_F} = x^B$. In particular, $x^B \notin \Z^n$, and $x^B$ is feasible whenever $x^{B_F}$ is. Consider first the Gomory fractional cut of $F$ from a row $i$ of $B_F$ with $\tilde b_i \notin \Z$. Then $i \in B$, and the Gomory fractional cut of $P$ from row $i$ of $B$ has the same coefficients $f_j/f_0$ for all $j \in N_F$. For the Dantzig cut of $F$ from $B_F$, the Dantzig cut of $P$ from $B$ has coefficients equal to $1$ on $N_F$. In each case, the cut of $P$ belongs to the corresponding family, and its coefficients on $N_F$ equal those of the cut of $F$.
\end{proof}

By Theorem~\ref{thm:corner-cut-hered}(a), the inclusions 
\begin{align}
    G_{B}(P) \cap F \subseteq G_{B}(F), \qquad G_{FB}(P) \cap F \subseteq G_{FB}(F), \qquad D_{FB}(P) \cap F \subseteq D_{FB}(F) \notag
\end{align}
hold for every proper face $F$. The reverse inclusions are precisely the ones that fail in the examples of Sections~\ref{subsec:counter-gfc} and~\ref{subsec:counter-dantzig}. By Theorem~\ref{thm:corner-cut-hered}(b), none of the three families satisfies the restriction condition.

\subsection{The Dantzig Closure from All Bases} \label{sec:dantzig}

We define the Dantzig closure derived from the cuts \eqref{def:dantzig-cut} associated with all bases,
\begin{align}
    D_{B}(P) := P \cap \Big\{x \in \R^n: \sum_{j \in N} x_j \geq 1 \text{ for all bases } B \text{ of } A \text{ such that } x^{B} \notin \Z^n \Big\}. \notag
\end{align}
In the notation of Section~\ref{subsec:corner-hered}, this closure is the $\Gamma_{\mathrm{DB}}$-closure of $P$.
Since every cut in \eqref{eq:dantzig-feasible-closure} is also generated by $\Gamma_{\mathrm{DB}}$, we have $D_{B}(P) \subseteq D_{FB}(P)$. We now show that the closure $D_{B}(P)$ satisfies the hereditary property.

\begin{lemma} \label{lemma:dantzig-compatible}
    $\Gamma_{\mathrm{DB}}$ is basis-compatible.
\end{lemma}

\begin{proof}
    For the lifting condition, consider the Dantzig cut $\sum_{j \in N_F} x_j \geq 1$ of $F$ from a basis $B_F$ of $A_F$, and let $B = B_F \setminus D$. By Lemma~\ref{lemma:basis-correspondence}, $x^B = x^{B_F} \notin \Z^n$, so $\Gamma_{\mathrm{DB}}(B)$ contains $\mathbf{1}$, which generates the Dantzig cut $\sum_{j \in N} x_j \geq 1$ with coefficients equal to $1$ on $N_F$. For the restriction condition, consider the Dantzig cut of $P$ from a basis $B$, not valid for $F$. By Lemma~\ref{lemma:basis-extension}, there exists a basis $B_F$ of $A_F$ with $B \subset B_F$ and $x^{B_F} \notin \Z^n$. Then $\Gamma_{\mathrm{DB}}(B_F)$ contains $\mathbf{1}$, which generates the Dantzig cut $\sum_{j \in N_F} x_j \geq 1$ with coefficients again equal to $1$ on $N_F \subset N$.
\end{proof}

\begin{corollary}[Dantzig closure] \label{cor:dantzig-hered}
    Let $P = \{x \in \R_+^n : Ax = b\}$ be a rational polyhedron in standard form and let $F$ be a face of $P$. Then $D_{B}(F) = D_{B}(P) \cap F$.
\end{corollary}

\begin{proof}
    Immediate from Lemma~\ref{lemma:dantzig-compatible} and Theorem~\ref{thm:corner-cut-hered}, the case $F = P$ being trivial.
\end{proof}

Corollary~\ref{cor:dantzig-hered} and the examples of Section~\ref{sec:non-hered} exhibit an interesting contrast. For the Dantzig procedure, enlarging the family from the feasible bases to all bases brings about the hereditary property. The basis $B_F$ produced by Lemma~\ref{lemma:basis-extension} may be infeasible, in which case $\Gamma_{\mathrm{DFB}}(B_F)$ contains no cut and the restriction condition fails, whereas $\Gamma_{\mathrm{DB}}(B_F)$ always contains the Dantzig cut of $F$ from $B_F$. For the Gomory fractional procedure, no such repair is possible. The example of Section~\ref{subsec:counter-gfc} breaks the hereditary property for the closures derived from feasible bases and from all bases alike.

\section{Concluding Remarks} \label{sec:conclusion}

We have identified two mechanisms that account for the hereditary property in several classical cutting-plane procedures. When a single family of lattice-free convex sets realizes the closure of a convex set and of each of its faces, the property follows from Theorem~\ref{thm:L-closure-hered}, as for the split, lift-and-project, Lov\'asz--Schrijver, Sherali--Adams, and Lasserre closures. When the cuts are instead indexed by the bases of a system in standard form, no such common family is available, and the property follows from the compatibility conditions of Theorem~\ref{thm:corner-cut-hered}, as for the closure of Dantzig cuts derived from all bases. Neither mechanism is universal. The mixed-integer Chv\'atal and $+$-cut closures fall outside both. The closures of Gomory fractional cuts and of Dantzig cuts derived from feasible bases, by contrast, are defined by families of corner cuts that satisfy the lifting condition but not the restriction condition, so exactly one of the two inclusions holds.

The Chv\'atal closure occupies a curious position with respect to these frameworks. It satisfies the hereditary property in the pure integer polyhedral setting (\cite{schrijver1980cutting}), but neither mechanism accounts for this. Chv\'atal inequalities are derived from the inequalities describing the polyhedron, so no common family of lattice-free convex sets is available, and the natural basis-indexed relaxation, the closure of Gomory fractional cuts, is strictly larger in general (\cite{cornuejols2001elementary}) and fails the hereditary property by Example~\ref{ex:gfc}. Whether a structural criterion of the kind developed in this work can account for the Chv\'atal closure remains open, and such a criterion would have to distinguish the pure integer from the mixed-integer setting, where the property fails. It is also natural to ask whether the closure obtained from all nontrivial valid inequalities for $\operatorname{corner}(B)$, over all bases $B$, satisfies the hereditary property.

\bigskip
{\bf Acknowledgments:} The authors thank Fritz Eisenbrand for raising the question addressed in this paper.

\bibliographystyle{plainnat}
\bibliography{references}

\end{document}